\documentclass[reqno]{amsart}

\usepackage[utf8]{inputenc}
\usepackage[T1]{fontenc}
\usepackage{mathtools}
\usepackage{amssymb}
\usepackage{stmaryrd}
\usepackage{amsthm,thmtools}
\usepackage{amsfonts}
\usepackage{graphicx}
\usepackage{enumerate}
\usepackage{enumitem}
\usepackage{tikz}
\usepackage{tikz-cd}
\usepackage{commath}
\usepackage{booktabs}
\usepackage[mathscr]{euscript}
\usepackage{stmaryrd}
\usepackage{hyperref}
\usepackage{parskip}
\usepackage{epigraph}
\usepackage[normalem]{ulem}         
\usepackage{amsrefs}
\usepackage[scr=boondox,  
            cal=esstix]   
           {mathalpha}
\usepackage{bbm}
\usepackage{dsfont}
\usepackage[all]{xy}
\usepackage[T1]{fontenc}
\usepackage{pgfplots}
\pgfplotsset{compat=1.15}
\usepackage{mathrsfs}
\usepackage{comment}

\title{Operator Algebras over the $p$-adic integers}

\theoremstyle{definition}
\newtheorem{theorem}{Theorem}[section]
\newtheorem{definition}[theorem]{Definition}
\newtheorem{example}[theorem]{Example}
\newtheorem{lemma}[theorem]{Lemma}
\newtheorem{proposition}[theorem]{Proposition}
\newtheorem{corollary}[theorem]{Corollary}

\newtheorem{remark}[theorem]{Remark}
\newtheorem*{theorem*}{Theorem}

\newcommand*{\nb}{\nobreakdash}

\newcommand{\und}[1]{\underline{#1}}

\newcommand{\defeq}{\mathrel{:=}} 

\newcommand{\Mm}{\mathfrak{m}}

\newcommand{\an}{\mathrm{an}}

\newcommand{\zZ}{\mathbb{Z}}
\newcommand{\nN}{\mathbb{N}}

\newcommand{\fF}{\mathbb{F}}

\newcommand*{\Z}{\mathbb Z}
\newcommand*{\N}{\mathbb N}
\newcommand*{\C}{\mathbb C}

\newcommand*{\sbe}{\subseteq}

\DeclarePairedDelimiterX{\setgiven}[2]{\{}{\}}{#1\,{:}\,\mathopen{}#2}
\newcommand*{\Mat}{\mathbb M}

\newcommand{\zP}{\mathbb{Z}_{p}}
\newcommand{\Zp}{\mathbb{Z}_{p}}
\newcommand{\qP}{\mathbb{Q}_{p}}

\newcommand{\OO}{\mathcal{O}}

\newcommand{\BB}{\mathcal{B}}
\newcommand{\BUN}{\mathcal{B}_{\leq 1}}

\newcommand{\GG}{\mathcal{G}}
\newcommand{\HH}{\mathcal{H}}

\newcommand{\fP}{\mathbb{F}_p}

\newcommand{\Cest}{C^{\ast}}

\newcommand{\Gcfin}{G_c^{\mathrm{fin}}}

\newcommand*{\braket}[2]{\langle#1\!\mid\!#2\rangle}

\makeatletter
    \def\@Obs#1{{\color{blue}{\uline{#1}}}}
    \def\Obs{\@ifstar{\@Obs}{\marginpar[\textcolor{blue}{\((\star)\)}]{\textcolor{blue}{\((\star)\)}}\@Obs}}
    
\makeatother

\begin{document}

\title{Operator algebras over the \(p\)-adic integers - II}
\author{Alcides Buss}
\email{alcides.buss@ufsc.br}
\address{Departamento de Matem\'atica\\
 Universidade Federal de Santa Catarina\\
 88.040-900 Florian\'opolis-SC\\
 Brazil}

\author{Luiz Felipe Garcia}
\email{lfgarcia98@gmail.com}
\address{Pós-Graduação em Matem\'atica\\
 Universidade Federal de Santa Catarina\\
 88.040-900 Florian\'opolis-SC\\
 Brazil}

\author{Devarshi Mukherjee}
\email{devarshi.mukherjee@maths.ox.ac.uk}
\address{University of Oxford\\ Mathematical Institute \\ Woodstock Rd\\ Oxford OX2 6GG, United Kingdom}

\begin{abstract}
We continue the study of operator algebras over the $p$-adic integers, initiated in our previous work \cite{BGM:padicI}. In this sequel, we develop further structural results and provide new families of examples. We introduce the notion of $p$-adic von Neumann algebras, and analyze those with trivial center, that we call ``factors''. In particular we show that ICC groups provide examples of factors. We then establish a characterization of $p$-simplicity for groupoid operator algebras, showing its relation to effectiveness and minimality. A central part of the paper is devoted to a $p$-adic analogue of the GNS construction, leading to a representation theorem for Banach $^*$-algebras over $\mathbb{Z}_p$. As applications, we exhibit large classes of $p$-adic operator algebras, including residually finite-rank algebras and affinoid algebras with the spectral norm. Finally, we investigate the $K$-theory of $p$-adic operator algebras, including the computation of homotopy analytic \(K\)-theory of continuous \(\zP\)-valued functions on a compact Hausdorff space and the analytic (non-homotopy invariant) \(K\)-theory of certain \(p\)-adically complete Banach algebras in terms of continuous \(K\)-theory. Together, these results extend the foundations of the emerging theory of $p$-adic operator algebras.
\end{abstract}

\thanks{The first author is partially supported by CNPq and Fapesc - Brazil. The third named author was funded by a DFG Eigenestelle (project number 534946574) and the Deutsche Forschungsgemeinschaft (DFG, German Research Foundation) under Germany's Excellence Strategy EXC 2044390685587, Mathematics Münster: Dynamics Geometry Structure and a UK Research and Innovation Horizon Europe Guarantee MSCA Postdoctoral Fellowship award.}



\maketitle


\section{Introduction}

The theory of operator algebras has played a central role in functional analysis, noncommutative geometry, and representation theory. In the classical (Archimedean) setting, $C^*$-algebras and von Neumann algebras provide a rich framework that connects group representations, noncommutative topology, and index theory. Motivated by recent developments in non-Archimedean functional analysis and $p$-adic geometry, a parallel theory of operator algebras over $\mathbb{Z}_p$ and $\mathbb{Q}_p$ has been initiated.  

In our previous paper entitled `\emph{Operator algebras over the $p$-adic integers}'~\cite{BGM:padicI}, we introduced the basic definitions of $p$-adic operator algebras, established their first structural properties, and provided initial examples such as groupoid algebras, rotation algebras, and $p$-adic analogues of Cuntz algebras. We also developed an enveloping construction for involutive algebras over $\Z_p$, and studied basic $K$-theory, both in its homotopy analytic and continuous forms. These results laid the foundations for the subject and opened up several natural directions of further study.  

The present article continues this project, expanding the foundations and exploring new structural aspects of the theory. In particular, we address questions about bicommutants and centers, simplicity, representations via quasi-Hilbert spaces, and refinements of $K$-theory.  

Our contributions can be summarized as follows:
\begin{itemize}
    \item \textbf{$p$-adic von Neumann algebras.} We introduce a $p$-adic analogue of von Neumann algebras and analyze their centers, obtaining a description in terms of finite conjugacy classes of groups.
    \item \textbf{$p$-simplicity.} For ample Hausdorff groupoids, we characterize minimality and effectiveness in terms of $p$-simplicity of the associated operator algebra.
    \item \textbf{A $p$-adic GNS theorem.} We construct quasi-Hilbert spaces and prove a $p$-adic version of the GNS representation theorem for Banach $^*$-algebras over $\mathbb{Z}_p$. This yields broad classes of examples, including residually finite-rank and affinoid algebras.
    \item \textbf{$K$-theory.} We investigate both homotopy analytic and continuous $K$-theory for commutative $p$-adic operator algebras. Using Clausen’s descent formalism, we compute the homotopy analytic $K$-theory of $C(X,\mathbb{Z}_p)$ in terms of $\Gamma(X,K(\mathbb{F}_p))$. We also describe continuous $K$-theory for general $p$-adic operator algebras in terms of nuclear modules, following recent work of Efimov and collaborators.
\end{itemize}

These results enrich the structure theory of $p$-adic operator algebras and extend the framework developed in~\cite{BGM:padicI}, opening new perspectives for further interactions with noncommutative geometry, representation theory, and $p$-adic analytic geometry.  

The paper is organized as follows. Section~\ref{sec:vn} introduces $p$-adic von Neumann algebras and studies their centers. Section~\ref{sec:psimple} deals with $p$-simplicity and groupoid algebras. Section~\ref{sec:gns} develops the $p$-adic GNS construction and its consequences. Section~\ref{sec:kh} computes the homotopy analytic $K$-theory of commutative $p$-adic operator algebras of the form $C(X,\zP)$ for $X$ a compact Hausdorff space, while Section~\ref{sec:nonhomotopy} discusses non-homotopy invariant $K$-theories and continuous refinements.  

Throughout this paper let $p$ be a prime number.

\section{$p$-adic von Neumann algebras}\label{sec:vn}

In the Archimedean theory, von Neumann algebras are characterized either as double commutants or as strong/weak operator closures of $C^*$-algebras on Hilbert spaces. In the non-Archimedean case, it is natural to ask for an analogue of this construction for operator algebras over $\mathbb{Z}_p$. The aim of this section is to propose a first definition of $p$-adic von Neumann algebras and to examine their centers in the case of group algebras.

Recall from \cite{BGM:padicI} that a $p$-adic operator algebra is a Banach $^*$-algebra $A$ over $\mathbb{Z}_p$ that admits an isometric representation as a $*$-subalgebra of $\BUN(\qP(X))$, where $Q_p(X)$ is the non-Archimedean analogue of $\ell^2(X)$ introduced in \cite{BGM:padicI} following Clausnitzer--Thom \cite{Clausnitzer-Thom:p-adic}. Here $X$ is an arbitrary set. One may identify $\BUN(\qP(X))$ with $\BB(c_0(X,\mathbb{Z}_p))$ and the space $\HH=c_0(X,\mathbb{Z}_p)$ is a basic example of a ``quasi-Hilbert space'' when equipped with its canonical $\mathbb{Z}_p$-valued pairing, see Section~\ref{sec:gns} for more details.  

In the present paper we introduce quasi-Hilbert spaces as a general framework for $p$-adic functional analysis (see Section~\ref{sec:gns}). Motivated by the bicommutant theorem in the Archimedean case, we adopt the following working definition.

Let $X$ be a set. Let $M$ be a subset of $\BUN(\qP(X))$; define
\begin{align*}
    M' = \set{a \in \BUN(\qP(X)) \mid ax = xa\,\mbox{ for all } x \in M}.
\end{align*}

\begin{definition}[von Neumann $p$-adic operator algebra] Let $X$ be a set. A \textit{$p$-adic von Neumann  algebra} is a $\ast$-subalgebra $M$ of $\BUN(\qP(X))$ such that
\begin{align*}
    M'' = M.
\end{align*}
\end{definition}

\begin{remark}
At this stage it is not clear whether a $p$-adic analogue of the bicommutant theorem holds, that is, whether $M=M''$ is equivalent to $M$ being closed in an appropriate weak or strong operator topology. Developing the correct non-Archimedean topologies on $\BB(\HH)$ is an open problem, which we hope to address in future work.
\end{remark}

\begin{example}
For every $\ast$-subalgebra $B \subseteq \BUN(\qP(X))$, the commutant $B'$ is a $p$-adic von Neumann  algebra i.e. $B' = B'''$. To see this let $b \in B'$. Note that
\begin{align*}
    b \in B''' &\iff \forall c \in B'': \ cb = bc .
\end{align*}
Let $c \in B''$; then for every $d \in B'$ we have $cd = dc$. In particular, $cb = bc$ from which we conclude that $b \in B'''$. Now let $b \in B'''$, then, for every $c \in B''$ we have $cb = bc$. We claim that $B \subseteq B''$, to prove this, take $d \in B$, by definition, for every $e \in B'$ we have $ed = de$ concluding the claim. The element $b$ commutes with everything in $B''$, therefore it commutes with everyone in $B$ and so $b \in B'$.

In particular, if $B\sbe \BUN(\qP(X))$ is a $p$-adic operator algebra, then $B''$ is a $p$-adic von Neumann algebra. 
\end{example}

Another source of examples are intersections:

\begin{proposition}\label{prop:intersection-infinite}
Let $\{M_i : i \in I\}$ be a family of $p$-adic von Neumann algebras acting on the same $p$-adic Hilbert space $\HH=\qP(X)$. Then their intersection
\[
M := \bigcap_{i\in I} M_i
\]
is again a $p$-adic von Neumann algebra on $\HH$; equivalently $M''=M$.
\end{proposition}

\begin{proof}
For arbitrary subsets $A_i \subseteq \BUN(\HH)$ one has the commutant identity
\[
\Big(\bigcup_{i\in I} A_i\Big)' = \bigcap_{i\in I} A_i'.
\]
Applying this to $A_i=M_i'$ gives
\[
\Big(\bigcup_{i\in I} M_i'\Big)' = \bigcap_{i\in I} M_i''.
\]
Since each $M_i$ is a von Neumann algebra, $M_i''=M_i$, hence the right-hand side equals $\bigcap_{i\in I} M_i = M$. Taking commutants on both sides, we obtain
\[
\Big(\bigcup_{i\in I} M_i'\Big)'' = M'.
\]
Taking commutants once more yields
\[
\Big(\big(\bigcup_{i\in I} M_i'\big)''\Big)' = M''.
\]
But $\big((\bigcup_i M_i')''\big)'=(\bigcup_i M_i')'=M$, so $M''=M$ as claimed.
\end{proof}

\begin{corollary}\label{cor:center-vn}
If $M$ be a $p$-adic von Neumann algebra on $\qP(X)$, then so is its commutant, and therefore also its \emph{center}
\[
Z(M) = M\cap M'.
\]
\end{corollary}

\begin{remark}
The proofs above are purely algebraic and rely only on the formal properties of the commutant operation and the bicommutant condition. In the Archimedean theory one often deduces the same facts from topological arguments (e.g.\ by observing that intersections of weakly closed *-algebras are weakly closed). In the $p$-adic setting a parallel topological viewpoint requires a careful development of weak/strong operator topologies on $\BUN(\HH)$; establishing the equivalence between bicommutant and closure formulations in that non-Archimedean context is an interesting direction for future work.
\end{remark}

Next we consider $p$-adic group algebras as another natural source of von Neumann algebras. 
The following example already appears in Claussnitzer's thesis \cite{claussnitzer-thesis}*{Theorem~4.1.1} for countable groups.

\begin{example}\label{ex:group-vn}
Let $G$ be a discrete group. Consider the $p$-adic group algebra $c_0(G,\mathbb{Z}_p)$, represented on $\mathbb{Q}_p(G)$ by the left convolution action
\[
\lambda\colon c_0(G,\mathbb{Z}_p)\to \BB(\mathbb{Q}_p(G)),\qquad 
\lambda(\xi)(\eta)=\xi\ast \eta.
\]
Denote by $\mathcal{O}_p(G):=\lambda(c_0(G,\mathbb{Z}_p))$ the image of this representation. 
By \cite{BGM:padicI} we know that $\lambda$ is an isometric $^*$-homomorphism, so $\mathcal{O}_p(G)\cong c_0(G,\mathbb{Z}_p)$ as Banach $^*$-algebras over $\zP$. 

We claim that $\mathcal{O}_p(G)$ is already a $p$-adic von Neumann algebra, i.e.\ $\mathcal{O}_p(G)=\mathcal{O}_p(G)''$. 
Indeed, an operator $S\in \BB(\mathbb{Q}_p(G))$ belongs to $\mathcal{O}_p(G)'$ if and only if
\[
\lambda(\delta_g)S=S\lambda(\delta_g)\quad\forall g\in G,
\]
which is equivalent to the matrix identity 
\[
S_{g^{-1}h,k}=S_{h,gk}\qquad\forall g,h,k\in G.
\]
Given $x,y\in G$, define $T(x,y)\in \BB(\mathbb{Q}_p(G))$ by 
\[
T(x,y)_{h,k}=
\begin{cases}
1,&\text{if }(h,k)=(gx,gy)\text{ for some }g\in G,\\
0,&\text{otherwise.}
\end{cases}
\]
One checks that each $T(x,y)\in\mathcal{O}_p(G)'$. If $S\in\mathcal{O}_p(G)''$, then $ST(x,y)=T(x,y)S$ for all $x,y\in G$, which forces 
\[
S_{h,k}=S_{hg,kg}\qquad\forall g,h,k\in G.
\]
Defining $\phi_S\colon G\to\mathbb{Z}_p$ by $\phi_S(x):=S_{x,e}$, it follows from \cite{BGM:padicI}*{Proposition~4.9} that $\phi_S\in c_0(G,\mathbb{Z}_p)$ and that $\lambda(\phi_S)=S$. Thus $S\in\mathcal{O}_p(G)$, and we conclude that $\mathcal{O}_p(G)''=\mathcal{O}_p(G)$. Hence the group algebra $c_0(G,\mathbb{Z}_p)\cong \OO_p(G)$ is a $p$-adic von Neumann algebra.
\end{example}

The above shows that, in the context of $p$-adic operator algebras, the group algebra $c_0(G,\mathbb{Z}_p)\cong\mathcal{O}_p(G)$ is already a $p$-adic von Neumann algebra. This stands in sharp contrast with the classical (Archimedean) situation: for a discrete group $G$, the complex group algebra $\mathbb{C}G$ admits several different completions as $C^*$-algebras, 
such as the full group $C^*$-algebra $C^*(G)$, its reduced counterpart 
$C^*_{r}(G)\subseteq\mathcal{B}(\ell^2(G))$, and various intermediate exotic completions. 
Moreover, the group von Neumann algebra $L(G)=C^*_r(G)''\subseteq \mathcal{B}(\ell^2(G))$ 
is distinct from $C^*_r(G)$ unless $G$ is finite. 
In the $p$-adic setting, by contrast, all these possible completions collapse to a single algebra, 
namely $c_0(G,\mathbb{Z}_p)\cong\mathcal{O}_p(G)$.

\begin{remark}
The result above does not extend to groupoid $p$-adic operator algebras, that is, to the algebras 
$\mathcal{O}_p(G)\subseteq \BUN(\qP(G))$ defined from an étale groupoid $G$, see \cite{BGM:padicI} for details. 
A first simple obstruction for an algebra to be a $p$-adic von Neumann algebra is that it must be unital. 
For instance, the algebra of compact operators 
$\mathcal K_{\leq 1}(\qP(X))\subseteq \BUN(\qP(X))$ 
is not a $p$-adic von Neumann algebra unless $X$ is finite, because its commutant is trivial 
(see Proposition~\ref{prop:commutant-compacts} below) and therefore its bicommutant is the whole $\BUN(\qP(X))$. 
Observe that $\mathcal K_{\leq 1}(\qP(X))$ is (isomorphic to) the groupoid algebra $\mathcal{O}_p(G)$ 
of the pair groupoid $G=X\times X$, see \cite{BGM:padicI}*{Example~6.19}. 
Even if we adjoin a unit, considering the unitization 
$\mathcal K_{\leq 1}(\qP(X))+\zP\cdot 1\subseteq \BUN(\qP(X))$, 
this is not a $p$-adic von Neumann algebra in general: 
$\mathcal K_{\leq 1}(\qP(X))''=\BUN(\qP(X))$ (see Proposition~\ref{prop:commutant-compacts} below), so that any 
$^*$-subalgebra $A\subseteq \BUN(\qP(X))$ containing $\mathcal K_{\leq 1}(\qP(X))$ 
necessarily satisfies $A''=\BUN(\qP(X))$.
\end{remark}

\begin{proposition}\label{prop:commutant-compacts}
Let $\mathcal B=\BUN(\qP(X))$. Let
\[
\mathcal K \subseteq \mathcal B
\]
be the subalgebra of compact operators (restricted to the unit ball), as defined in \cite{Clausnitzer-Thom:p-adic}.
Then the commutant of $\mathcal K$ inside $\mathcal B$ is trivial:
\[
\mathcal K' = \{ \lambda\cdot 1 : \lambda \in \mathbb Z_p \},
\]
that is, any $T\in B$ with $TK=KT$ for all $K\in\mathcal K$ is a scalar operator, and the scalar lies in $\mathbb Z_p$.
\end{proposition}

\begin{proof}
Denote by $(\delta_x)_{x\in X}$ the canonical delta-functions in $Q_p(X)$.
For $x,y\in X$ consider the elementary finite-rank operator $E_{x,y}\in B$ defined on the canonical basis by
\[
E_{x,y}(\delta_z) = \begin{cases} \delta_x & \text{if } z=y,\\ 0 & \text{if } z\neq y. \end{cases}
\]
Equivalently, the matrix of $E_{x,y}$ has the single nonzero entry $1$ at position $(x,y)$. Clearly every $E_{x,y}$ has norm $\le 1$ and hence $E_{x,y}\in\mathcal K$.

Let $T\in B$ satisfy $TE_{x,y}=E_{x,y}T$ for all $x,y\in X$. We compare the actions of both sides on a basis vector $\delta_z$:
\[
TE_{x,y}(\delta_z)=T(\delta_x)\cdot\mathbf{1}_{\{z=y\}},\qquad
E_{x,y}T(\delta_z)=T_{y,z}\,\delta_x,
\]
where $T_{u,v}$ denotes the matrix entries of $T$, i.e.\ $T(\delta_v)=\sum_{u} T_{u,v}\delta_u$.
Thus, for all $x,y,z\in X$,
\[
\mathbf{1}_{\{z=y\}}\,T(\delta_x)=T_{y,z}\,\delta_x.
\]

Fix arbitrary $x,y\in X$. If $z\neq y$ the left-hand side vanishes, hence the right-hand side must vanish; evaluating at the coordinate $\delta_x$ we obtain $T_{y,z}=0$ whenever $z\neq y$. Therefore for each fixed $y$ the only possibly nonzero entries in the \(y\)-th row of \(T\) occur at column \(z=y\); equivalently,
\[
T_{y,z}=0\quad\text{for }z\neq y.
\]
Now put $z=y$ in the displayed identity. For every $x,y$ we get
\[
T(\delta_x)=T_{y,y}\,\delta_x.
\]
The left-hand side does not depend on $y$, hence $T_{y,y}$ is independent of $y$. Write $\lambda:=T_{y,y}\in\mathbb Q_p$ (same for all $y$). Then
\[
T(\delta_x)=\lambda\,\delta_x\qquad(\forall x\in X),
\]
so \(T=\lambda\cdot 1\) is a scalar operator.

It remains to see that \(\lambda\in\mathbb Z_p\). Since \(T\) belongs to the closed unit ball \(B\) we have \(\|T\|\le 1\). But \(\|T\|=\sup_{\| \xi\|\le1}\|T\xi\|=\sup_{\| \xi\|\le1}\|\lambda\xi\|=|\lambda|_p\), therefore \(|\lambda|_p\le1\) and hence \(\lambda\in\mathbb Z_p\). This shows $\mathcal K' \subseteq \mathbb Z_p\cdot 1$.

Conversely, any scalar $\lambda\cdot 1$ with $\lambda\in\mathbb Z_p$ clearly commutes with every operator in $\mathcal K$, so $\mathcal K' = \mathbb Z_p\cdot 1$ as claimed.
\end{proof}

\subsection{Centers and factor $p$-adic von Neumann algebras}

\begin{definition}[Factor] Given a $p$-adic von Neumann algebra $M\sbe \BUN(\qP(X))$, we write $Z(M)=M\cap M'$ for its center. We say that $M$ is a \emph{factor} if its center is trivial, that is, $Z(M)=\zP\cdot 1$, only scalar multiples of the identity operator.
\end{definition}

One simple example of a factor is $\BUN(\qP(X))$ itself, as is easy to check that its center is trivial. 

Next we are going to describe when group $p$-adic operator algebras lead to factors.

\begin{proposition} Let $G$ be a discrete group. We then have 
\begin{align*}
    Z(c_0(G,\zP)) = \set{\phi \in c_0(G, \zP) \mid \forall g, h \in G: \ \phi(g^{-1}h) = \phi( hg^{-1}) }.
\end{align*}
\end{proposition}
\begin{proof}
Let $\phi \in Z(c_0(G,\zP))$. For each $g_0 \in G$, the delta function $\delta_{g_0} \in c_0(G,\zP)$, and we have
\begin{align*}
    \phi \ast \delta_{g_0} = \delta_{g_0} \ast \phi.
\end{align*}
It is easy to see that for an arbitrary element $h \in G$, we have \(\phi(g_0^{-1}h) = \phi(hg_0^{-1})\), so that \begin{align*}
    Z(c_0(G,\zP)) \subseteq \set{\phi \in c_0(G, \zP) \mid \forall g, h \in G: \ \phi(g^{-1}h) = \phi(hg^{-1}) }.
\end{align*}

Now suppose $\phi \in c_0(G, \zP)$ is such that \(\phi(g^{-1}h) = \phi(hg^{-1})\). Then for an arbitrary \(\psi \in c_0(G,\zP)\) and \(h \in G\), we get
\begin{align*}
    (\phi \ast \psi)(h) &= \sum_{h \in G} \phi(k^{-1}h) \psi(k) = \sum_{g \in G} \phi(g^{-1} h) \psi(g) = \sum_{g \in G} \phi(hg^{-1}) \psi(g) \\
    &= \sum_{g \in G} \psi(g) \phi(hg^{-1}) = (\psi \ast \phi)(h),
\end{align*}
concluding that $\phi \in Z(c_0 (G, \zP))$
\end{proof}

Given $g\in G$, we denote by $C(g):=\{hgh^{-1}\in G: h\in G \}$ the conjugacy classes of $g$ in $G$. The above proposition says that $F\in Z(c_0(G,\zP))$ if and only if $F$ is constant on conjugacy classes. Notice that the conjugacy classes are the orbits of the conjugation action of $G$ on itself: the action given by $c_g(h):=ghg^{-1}$. In particular, $hC(g)h^{-1}=C(g)$. Therefore, $F$ is central in $c_0(G,\zP)$ if and only if it induces a function on the orbit space $G/c:=\{C(g): g\in G\}$.

Let $\Gcfin:=\{g\in G: |C(g)|<\infty\}$. Then $\Gcfin$ is a normal subgroup of $G$.

How can we relate the center of $c_0(G,\zP)$ to $\Gcfin$?

\begin{corollary}
    The center of $c_0(G,\zP)$ is trivial, that is, $Z(c_0(G,\zP))=\zP$ iff $\Gcfin$ is the trivial group, that is, $G$ is an i.c.c group -- the non-trivial conjugacy classes are all infinite: $|C(g)|=\infty$ for all $g\not=e$.
\end{corollary}
\begin{proof}
    Suppose $G$ is not i.c.c, and take $e\not= g\in G$ with $|C(g)|<\infty$. Define $\phi:=\sum_{h\in C(g)} \delta_h$, that is, $\phi(h)=1$ if $h\in C(g)$ and $\phi(h)=0$ otherwise. Then $\delta_k \phi \delta_{k^{-1}}=\sum_{h\in C(g)}\delta_{khk^{-1}}=\phi$ so that $\phi\in Z(c_0(G,\zP))$, and therefore the center is non-trivial.

    Now assuming that $G$ is i.c.c. we show that $Z(c_0(G,\zP))=\zP$. If $G$ is the trivial group, we have nothing to do: in this case $c_0(G,\zP)=\zP$. If $e\not= g\in G$, then $|C(g)|=\infty$ by assumption. If $\phi\in Z(c_0(G,\zP))$, then $\phi$ is constant on $C(g)$ by the previous proposition.
    Since $C(g)$ in infinity and $\phi$ vanishes at infinity, this can only happen if $\phi$ is zero on $C(g)$.
    But this implies in particular that $\phi(g)=0$. Therefore $\phi$ is a multiple of $\delta_e$, which is the unit of $c_0(G,\zP)$, that is, $\phi\in \zP=\zP\cdot 1$, as desired.
\end{proof}

Now consider the family $\set{C_i}_{i \in I}$ of all finite conjugacy classes of $G$. Define for each $i \in I$,
\begin{align*}
    \chi_i = \sum_{g \in C_i} \delta_g.
\end{align*} It is easy to see that $\set{\chi_i \mid i \in I}$ is a linearly independent set of maps in $Z(c_0(G, \zP))$. 

\begin{proposition} 
We have  \(Z(c_0(G, \zP)) = \overline{\mathrm{span}} \set{\chi_i \mid i \in I}\).
\end{proposition}
\begin{proof}
Let $\phi \in Z(c_0(G, \zP))$.
Note that $\chi_i$ is a class function and therefore $\chi_i \in Z(c_0(G, \zP))$. There are at most countably many $g \in G$ such that $|\phi(g)| \neq 0$. To see this, note that for each $n \in \nN$, there are at most finitely many $g$'s with $|\phi(g)| > p^{-n}$. As $\phi$ is constant on the conjugacy classes we can set
\begin{align*}
    K = \set{i \in I \mid \text{for } g \in C_i: \phi(g) \neq 0}.
\end{align*}
Then $K$ is a countable set, and setting $a_i\in \zP$ to be the constant value of $\phi$ at $C_i$, we get
\begin{align*}
    \phi = \sum_{i \in K} a_i \chi_i.
\end{align*}
Writing \(K\) as a filtered union of finite sets \(K = \bigcup_{n \in \N} K_n\), we may express \(\phi\) as a uniform limit of finite partial sums
\begin{align*}
    \phi = \lim_{n \to \infty} \sum_{i \in K_n} a_i \chi_i = \sum_{i \in K} a_i \chi_i,
\end{align*}
as required.
\end{proof}

\begin{remark}
It remains open to determine to what extent a classification of ``$p$-adic factors'' is possible. The non-Archimedean setting may present new phenomena: for instance, it is not clear whether idempotents behave analogously to projections, and whether the classical type classification (I, II, III) can be carried over. We leave these questions for future work.
\end{remark}

\section{$p$-Simplicity}\label{sec:psimple}

\begin{definition} A $p$-adic operator algebra $A$ is said to be \emph{$p$-simple} if the algebraic quotient $A/pA$ is simple as an $\fP$-algebra. 
\end{definition}

Just as we can check that a Hausdorff ample groupoid is effective and minimal by looking at its $C^*$-algebra, one can also check this by looking at its $p$-adic operator algebra.

\begin{proposition} Let $\GG$ be a Hausdorff ample groupoid. Its $p$-adic group algebra $\OO_p(\GG)$ is $p$-simple if and only if $\GG$ is effective and minimal.
\end{proposition}
\begin{proof}
Let $R$ be a ring. Let $R\GG$ be the $R$-algebra spanned by the characteristic functions $\chi_U$ with $U$ being compact open local bisections. We define the product in this algebra by the convolution
\begin{align*}
    \phi \ast \psi (g) = \sum_{s(h) = s(g)} \phi(gh^{-1})\psi(h)
\end{align*}
It is well-known that this sum is finite.

The canonical projection $\pi: \zP \to \fP$ induces a surjective morphism
\begin{align*}
    \tilde{\pi} : \zP \GG &\twoheadrightarrow \fP\GG \\
    \chi_U &\mapsto \chi_U
\end{align*}
whose kernel is the principal ideal $p \zP \GG$. Consider $\fP \GG$ with the discrete topology and $\zP \GG$ with the sup norm. To see that $\tilde{\pi}$ is continuous, let $\psi \in \tilde{\pi}^{-1}(\phi)$. Consider the open neighborhood $V = \BB(\psi, \frac{1}{p})$, then for every $\theta \in V$ the element $\theta - \psi$ is a multiple of $p$, hence 
\begin{align*}
    \tilde{\pi}(\theta) = \tilde{\pi}(\psi) = \phi
\end{align*}
and $\psi \in V \subseteq \tilde{\pi}^{-1}(\phi)$, showing that $\tilde{\pi}^{-1}(\phi)$ is open. 

The algebra $\zP \GG$ is dense in $\OO_p(\GG)$ so that any element $\phi \in \OO_p (\GG)$ may be written as a limit $\phi = \lim_n \phi_n$ with $\phi_n \in \zP \GG$. As the sequence $(\phi_n)_n$ is Cauchy, there exists $n_0 \in \nN$ such that for every $n, m > n_0$, we have
\begin{align*}
    ||\phi_n - \phi_m|| < \frac{1}{p}.
\end{align*}
Consequently, the difference $\phi_n - \phi_m$ is a multiple of $p$ and
    \(\tilde{\pi}(\phi_n) = \tilde{\pi}(\phi_m)\), using which we may conclude that the sequence $( \tilde{\pi}(\phi_n))_n$ is eventually constant. Define
\begin{align*}
    \overline{\pi}(\phi) = \lim_n \tilde{\pi}(\phi_n), 
\end{align*}
which is continuous by definition. The map $\overline{\pi}$ extends $\tilde{\pi}$, and it is easy to see that its kernel is \(p \OO_p(\GG)\). Consequently, we have 

\begin{align*}
    \frac{\OO_p(\GG)}{p\OO_p(\GG)} \cong \fP \GG,
\end{align*}
and by \cite{Steinberg:Simplicity}*{Theorem 3.5}, this happens if and only if $\GG$ is effective and minimal.
\end{proof}

\section{A $p$-adic GNS Theorem}
\label{sec:gns}

Recall that following \cite{BGM:padicI}*{Corollary~4.5}, we have an isomorphism
\begin{align*}
    \BB_{\leq 1}(\qP(X)) \cong \BB(c_0(X,\zP)).
\end{align*}
Consequently, a Banach $\ast$-algebra over $\zP$ is a $p$-adic operator algebra if and only if it can be represented isometrically on $c_0(X,\zP)$ carrying the canonical pairing, that is, the pairing given by
$$\braket{\xi}{\eta}=\sum_{x\in X}\xi(x)\eta(x).$$
This indicates that the spaces $c_0(X,\zP)$ could also serve as models for $p$-adic Hilbert spaces. While this might not be the final definition, we propose in this section a general definition of what we call a `quasi-Hilbert space' (over the $p$-adic integers), and prove an analogue of the GNS-representation theory of Banach $*$-algebras over $\zP$. Along the way, we will show that if we expand the range of pairings on $c_0(X,\zP)$, we can represent virtually every Banach $\ast$-algebra (see Theorem \ref{thm:GNS}) on a quasi-Hilbert space.

Recall that a $\zP$-module $M$ is \emph{$p$-adically complete} if the canonical map
\[
M \longrightarrow \varprojlim_{n} M/p^nM
\]
is an isomorphism. The following theorem shows that the structure of any discretely valued, \(p\)-adically complete Banach \(\zP\)-module is particularly simple:

\begin{theorem}\label{thm:everyone-is-c0}
Let $M$ be a $p$-adically complete Banach $\zP$\nb-module whose norm takes values in $|\zP|$ (i.e.\ $\|M\|\subseteq|\zP|$). Then there exists a set $X$ and an isometric isomorphism of Banach $\zP$-modules
\[
M \cong c_0(X,\zP).
\]
\end{theorem}

\begin{proof}
This is essentially \cite{Cortinas-Meyer-Mukherjee:NAHA}*{Theorem 2.4.2}. Set $pM=\{pm:m\in M\}$ and consider the $\fF_p$-vector space $V:=M/pM$. Choose a basis $\{\bar\omega_x:x\in X\}$ of $V$ and lift each $\bar\omega_x$ to $\omega_x\in M$ (so necessarily $\|\omega_x\|=1$). Define
\[
\Phi\colon c_0(X,\zP)\to M,\qquad \Phi(\xi)=\sum_{x\in X}\xi(x)\omega_x.
\]
The map is well-defined because $\xi\in c_0(X,\zP)$ has values tending to $0$, hence the sum converges in $M$ by ultrametricity. One checks that $\Phi$ is $\zP$-linear and isometric: the ultrametric inequality yields $\|\Phi(\xi)\|\le\|\xi\|$, while the choice of lifts implies that if $\|\xi\|=1$ then $\Phi(\xi)\not\in pM$, hence $\|\Phi(\xi)\|=1$. Surjectivity follows from $p$-adic completeness: every $m\in M$ can be written (uniquely) as a convergent $p$-adic series in the lifted basis elements, so $m$ lies in the range of $\Phi$. This gives the claimed isometric identification.
\end{proof}

\begin{remark}\label{rem:divisibility-and-norms}
If $M\cong c_0(X,\zP)$ is a discretely valued $p$-adically complete Banach $\zP$-module, then:

\begin{enumerate}
    \item An element $\xi\in M$ satisfies $\|\xi\|<1$ if and only if $\xi$ is divisible by $p$, i.e.\ $\xi=p\eta$ for some $\eta\in M$. In this case one may unambiguously write $p^{-1}\xi:=\eta\in M$.
    
    \item The normalisation of an element \(\xi \in M\) is given by \(\eta = p^{-\nu_p(\xi)}\xi\) (rather than $\xi/\|\xi\|$ as in the Archimedean case), where \(\nu_p(\xi)\) is the \(p\)-adic valuation of \(\xi\).
\end{enumerate}
\end{remark}

\begin{definition}
A \textit{quasi-Hilbert space} is a $p$-adically complete Banach $\zP$-module $\HH$ with $\|\HH\|\subseteq|\zP|$, together with a nondegenerate symmetric bilinear form $\langle -,- \rangle$ such that
\[
\sup_{\|\eta\|\leq 1}\,|\langle \eta,\xi\rangle| \;=\;\|\xi\|
\]
for every $\xi\in\HH$. 
\end{definition}

The pairing associated with a quasi-Hilbert space $\HH$ satisfies the Cauchy–Schwarz inequality and is therefore continuous. Indeed, take $\xi,\eta\in\HH$ with $\xi\neq 0$ and denote by $\eta = p^{-\nu_p(\xi)}\xi$ the normalisation. We then have
\[
|\langle \xi,\eta\rangle|
= |p^n \langle \zeta,\eta\rangle|
\leq |p^n|\,\|\zeta\|\,\|\eta\|
= p^{-n}\|\eta\|
= \|\xi\|\,\|\eta\|.
\]
Thus the $p$-adic analogue of the Cauchy–Schwarz inequality holds.

\begin{example}
A basic example is $\HH=c_0(X,\zP)$, the Banach $\zP$-module of $\zP$-valued functions on $X$ vanishing at infinity, 
endowed with the supremum norm and the canonical pairing
\[
\braket{\xi}{\eta}:=\sum_{x\in X}\xi(x)\eta(x).
\]
By Theorem~\ref{thm:everyone-is-c0}, every quasi-Hilbert space is isometrically isomorphic, as a Banach $\zP$-module, 
to some $c_0(X,\zP)$, although the pairing may differ from the canonical one above. 

With respect to the canonical pairing, the functions $(\delta_x)_{x\in X}$ form an ``orthonormal basis'': each $\delta_x$ has norm one, 
\[
\braket{\delta_x}{\delta_y}=0\quad(x\neq y), \qquad 
\braket{\delta_x}{\delta_x}=1.
\]
Moreover, $(\delta_x)_{x\in X}$ is an ``orthonormal basis'' in the sense that their $\zP$-linear span is dense in $c_0(X,\zP)$, and every element 
$\xi\in c_0(X,\zP)$ admits the expansion
\[
\xi=\sum_{x\in X}\braket{\xi}{\delta_x}\,\delta_x.
\]
\end{example}

It is not clear whether arbitrary quasi-Hilbert spaces admit orthogonal bases. However, in the finite rank case such bases do exist, as we shall see in Theorem~\ref{theo:finite-rank-q-Hilb}.

\begin{definition}
Let $\HH$ be a quasi-Hilbert space. For a linear operator $T\colon \HH \to \HH$ we define its operator norm by
\[
\|T\| := \sup_{\|\xi\|\leq 1}\,\|T(\xi)\|.
\]
We denote by $\BB(\HH)$ the Banach $\ast$-algebra (over $\zP$) consisting of all bounded $\zP$-linear operators on $\HH$ which are 
\emph{adjointable} with respect to the pairing of $\HH$. 
Thus $T\in \BB(\HH)$ if and only if $T$ is $\zP$-linear, bounded, and there exists $T^*\in \BB(\HH)$ with
\[
\braket{\xi}{T(\eta)}=\braket{T^*(\xi)}{\eta}\qquad\forall\,\xi,\eta\in \HH.
\]
\end{definition}

We shall now prove a $p$-adic analogue of the GNS construction, replacing Hilbert spaces and states by quasi-Hilbert spaces and \emph{quasi-states}: 

\begin{definition}[Quasi-states]
A \emph{quasi-state} on a $\ast$-Banach algebra $A$ is a continuous $\zP$-linear functional 
$\phi\colon A\to \zP$ such that:
\begin{enumerate}
    \item $\phi(a^{\ast})=\phi(a)$ for all $a\in A$;
    \item $\|\phi\|=1$.
\end{enumerate}
\end{definition}

\begin{remark}
The above definition is motivated by the fact that on a (complex) unital C$^*$-algebra $A$, 
a state can equivalently be defined as a bounded linear functional $\phi\colon A\to \C$ with 
$\|\phi\|=\phi(1)=1$. In our setting, however, we drop the condition $\phi(1)=1$, even if the algebra is unital. 
For instance, the functional $\phi(\lambda):=-\lambda$ on $A=\C$ is a ``quasi-state'' 
which is not positive, and hence not a state in the usual sense.

More generally, if we allow arbitrary (complex) unital Banach $*$-algebras, then even 
imposing $\|\phi\|=\phi(1)=1$ does not necessarily imply positivity 
($\phi(a^*a)\geq 0$ for all $a\in A$). A simple example is 
$A=\C\times\C$ with the coordinatewise product, the sup-norm, and the involution 
$(\alpha,\beta)^*:=(\bar\beta,\bar\alpha)$. The functional 
$\phi(\alpha,\beta):=\tfrac{\alpha+\beta}{2}$ satisfies $\phi(a^*)=\overline{\phi(a)}$ and 
$\|\phi\|=\phi(1)=1$, but it is not positive. 
Similar remarks apply to real Banach $*$-algebras.

Thus, a quasi-state should be regarded as a natural weakening of the usual notion of state, dropping positivity -- which has no good analogue in the 
$p$-adic setting -- in favor of the symmetry condition $\phi(a^*)=\phi(a)$.
\end{remark}

Let us begin with a class of Banach $^*$-algebras that admit a natural representation on quasi-Hilbert spaces.

\begin{definition}[Quasi-$C^*$-algebra]
A \emph{quasi-$C^*$-algebra over $\zP$} is a \(p\)-adically complete Banach $^*$-algebra $A$ such that:
\begin{itemize}
    \item the norm of $A$ takes values in $|\zP|$, i.e.\ $\|a\|\in|\zP|$ for all $a\in A$;
    \item the following condition holds:
    \begin{description}
        \item[(quasi-C*-property)] For every $a\in A$, there exists a quasi-state $\phi$ on $A$ such that
        \[
        \sup_{\|b\|,\|c\|\leq 1}\,|\phi(b^*ac)|=\|a\|.
        \]
    \end{description}
\end{itemize}
\end{definition}

\noindent
Not every Banach $^*$-algebra satisfies this property. The simplest counterexamples are Banach $^*$-algebras with zero multiplication ($ab=0$ for all $a,b\in A$), where the supremum above is always $0$ even if $\|a\|=1$. Finding nontrivial unital counterexamples is more subtle; an explicit example will be given below (see Remark~\ref{rem:not-quasi-C}).

\begin{remark}
The above quasi-\(C^*\)-property is motived by the following situation in the complex setting: for a complex \(C^*\)-algebra $A$ and $a\in A$,
\[
\|a\|=\sup_{\substack{\varphi\ \text{state on }A\\ \|b\|,\|c\|\le 1}} |\varphi(b^*ac)|.
\]
The proof proceeds by passing to the universal representation $\pi_{\mathrm{univ}}=\bigoplus_{\varphi\in S(A)}\pi_\varphi$, for which $\|\pi_{\mathrm{univ}}(a)\|=\|a\|$, and then selecting a summand $\pi_\varphi$ that approximates the norm. In the GNS-representation $(\pi_\varphi,\HH_\varphi,\xi_\varphi)$ the cyclicity of $\xi_\varphi$ implies that vectors of the form $\pi_\varphi(b)\xi_\varphi$ are dense in $\HH_\varphi$, so one can approximate unit vectors by such vectors and obtain
\[
\|a\|\approx|\langle\pi_\varphi(a)\pi_\varphi(c)\xi_\varphi,\pi_\varphi(b)\xi_\varphi\rangle|
=|\varphi(b^*ac)|.
\]
Letting the approximations tend to zero yields the claim. 
\end{remark}

Every quasi-state in a Banach $\ast$-algebra induces a $\ast$-representation on a quasi-Hilbert space.

\begin{theorem}[Representation associated to a quasi-state]\label{theo:correctedGNS001} Let $A$ be a unital Banach $\ast$-algebra which is $p$-adically complete, and let $\phi$ a quasi-state in $A$. Then there exists:
\begin{enumerate}
    \item a quasi-Hilbert space $\HH$ and a $\ast$-representation
    \begin{align*}
        \pi_{\phi} \colon A \to \BB(\HH_{\phi}).
    \end{align*}
    \item a cyclic vector $\xi_{\phi} \in \HH_{\phi}$ such that 
    \begin{align*}
        \phi(a) = \langle \xi_{\phi}, \pi_{\phi}(a)(\xi_{\phi}) \rangle_{\phi}.
    \end{align*}
\end{enumerate}
\end{theorem}
\begin{proof}
Consider the left ideal \(I_{\phi} = \set{a \in A \mid \phi(b^{\ast}a) = 0 \ \forall b \in A}\) and the quotient $\zP$-module $A/I_{\phi}$. Then the assignment
\begin{align*}
    \|a + I_{\phi}\|_{\phi} = \sup_{\|b\| \leq 1} |\phi(b^{\ast}a)|
\end{align*} is a well-defined ultrametric norm. Indeed, if $a_0, a_1 \in A$ satisfy that $a_0 - a_1 \in I_{\phi}$, then $\phi(b^{\ast}a_0) = \phi(b^{\ast}a_1)$ for every $b \in A$.

If $\|a + I_{\phi}\|_{\phi} = 0$, then $\phi(b^{\ast}a) = 0$ for every $b$ with $\|b\| \leq 1$, but this implies that $a \in I_{\phi}$.

The property that $\|\lambda(a + I_{\phi})\|_{\phi} = |\lambda| \|a + I_{\phi}\|_{\phi}$ for all $a \in A$ and $\lambda \in k$ is immediate.

Let $a_0, a_1 \in A$, then
\begin{align*}
    \|a_0 + a_1 + I_{\phi}\|_{\phi} &= \sup_{\|b\| \leq 1} |\phi(b^{\ast}(a_0 + a_1))| \\
    &= \sup_{\|b\| \leq 1} |\phi(b^{\ast}a_0) + \phi(b^{\ast}a_1)| \\
    &\leq \sup_{\|b\| \leq 1} \max\set{|\phi(b^{\ast}a_0)|, |\phi(b^{\ast}a_1)|} \\
    &\leq \max\set{\sup_{\|b\| \leq 1} |\phi(b^{\ast}a_0)|, \sup_{\|b\| \leq 1} |\phi(b^{\ast}a_1)|} \\
    &= \max{\set{\|a_0 + I_{\phi}\|_{\phi}, \|a_1 + I_{\phi}\|_{\phi}}}
\end{align*}
which proves the ultrametric inequality.

Let $\HH_{\phi}$ be the completion of $A/I_{\phi}$ with the norm $\|\cdot\|_{\phi}$. In the quotient $A/I_{\phi}$ one can define the following symmetric nondegenerate bilinear form:
\begin{align*}
    \langle b + I_{\phi}, a + I_{\phi} \rangle_{\phi} = \phi(b^{\ast}a).
\end{align*}
Since $\phi$ is continuous and $\zP$ is complete, $\phi$ sends convergent sequences to convergent sequences, therefore one can extend $\langle \cdot, \cdot \rangle_\phi$ uniquely to $\HH_{\phi}$.

To see that the Banach module \(\HH_\phi\) is \(p\)-adically complete, let $a + I_{\phi} \in A/I_{\phi}$ be an element with $\|a + I_{\phi} \| < 1$. Then there exists an element $b \in I_{\phi}$ such that $\|a + b\| < 1$. As $A$ is $p$-adically complete, there exists an element $c \in A$ with $a + b = pc$. Thus $a + I_{\phi} = p(c + I_{\phi})$. We conclude that the quotient norm in $\HH_{\phi}$ is the $p$-adic norm and therefore is $p$-adically complete as the completion of a $p$-adically complete normed $\zP$-module is $p$-adically complete.

We now construct the canonical representation associated to \(\phi\). For every $a \in A$, define
\begin{align*}
    \pi_{\phi}(a) \colon \HH_{\phi} &\to \HH_{\phi}, \\
    c + I_{\phi} &\mapsto ac + I_{\phi}.
\end{align*}
We have
\begin{align*}
    \sup_{\|c\| \leq 1} \|\pi_{\phi}(a)(c)\|_{\phi} = \sup_{\|c\| \leq 1} \|ac + I_{\phi}\|_{\phi} = \sup_{\|c\| \leq 1} \sup_{\|b\| \leq 1} |\phi(b^{\ast}ac)| \leq \|a\| 
\end{align*}
where the last inequality comes from the fact that $\phi$ is contractive. This way we conclude that $\pi_{\phi}(a)$ is a bounded operator on $\HH_{\phi}$ and that the association
\begin{align*}
    \pi_{\phi} \colon A &\to \BB(\HH_{\phi}), \\
    a &\to \pi_{\phi}(a),
\end{align*}
is a contractive linear map. It is also clearly multiplicative and involutive, so a $\ast$-representation.

The cyclic vector mentioned is
\begin{align*}
    \xi_{\phi} = 1 + I_{\phi}.
\end{align*}
Note that it satisfies
\begin{align*}
    \phi(a) = \langle \xi_{\phi}, \pi_{\phi}(a)(\xi_{\phi}) \rangle_{\phi}.
\end{align*}\vskip -1,5pc
\end{proof}

The following proposition shows that the direct sum of quasi-Hilbert spaces can be naturally considered.

\begin{proposition} 
Let \(\{\HH_i = (c_0(X_i, \zP), \langle \cdot , \cdot \rangle_i)\}_{i \in I}\) be a family of quasi-Hilbert spaces. Define
\[
\HH = \bigoplus_{i \in I} \HH_i := \left(c_0\left(\bigsqcup_{i \in I} X_i, \zP\right), \langle \cdot , \cdot \rangle \right)
\quad \text{with} \quad
\langle f, g \rangle = \sum_{i \in I} \langle f|_{X_i}, g|_{X_i} \rangle_i.
\]
Then \(\HH\) is a quasi-Hilbert space.
\end{proposition}

We denote the quasi-Hilbert space constructed above by \(\bigoplus_i \HH_i = \HH\).

Finally, we are ready to discuss the first part of the main result of this section.

\begin{theorem}[p-adic GNS: First part] 
Let \(A\) be a unital quasi-\(\Cest\)-algebra. Then there exist a quasi-Hilbert space \(\HH\) and an isometric \(\ast\)-representation
\[
\pi \colon A \to \BB(\HH).
\]
\end{theorem}

\begin{proof}
Define
\[
\HH = \bigoplus_{\phi \text{ quasi-state}} \HH_{\phi},
\]
where \(\HH_{\phi}\) is the quasi-Hilbert space associated with the quasi-state \(\phi\).

Consider the map
\[
\pi \colon A \to \BB(\HH), \quad a \mapsto (\pi_{\phi}(a))_{\phi}.
\]

It is clear that \(\pi\) is a contractive \(\ast\)-representation; it remains to show that \(\pi\) is isometric.

Fix \(a \in A\). By the quasi-C*-property, there exists a quasi-state \(\phi\) such that
\[
\begin{aligned}
\|a\| &= \sup_{\|b\|, \|c\| \leq 1} |\phi(b^{\ast} a c)| \\
&= \sup_{\|b\|, \|c\| \leq 1} \left|\langle b + I_{\phi}, ac + I_{\phi} \rangle_{\phi}\right| \\
&= \sup_{\|b\|, \|c\| \leq 1} \left|\langle b + I_{\phi}, \pi_{\phi}(a)(c) \rangle_{\phi}\right| \\
&= \sup_{\|c\| \leq 1} \|\pi_{\phi}(a)(c)\| \\
&= \|\pi_{\phi}(a)\|.
\end{aligned}
\]

This concludes the proof.
\end{proof}

To complete our main theorem, we show that every Banach $^*$-algebra $A$ whose norm takes values in $|\Z_p|$ (that is, $\|A\|\subseteq|\Z_p|$) can be represented isometrically as a $^*$-subalgebra of some quasi-C*-algebra. It will then follow that every such Banach $*$-algebra is isomorphic to a subalgebra of $\BB(\HH)$ for some quasi-Hilbert space $\HH$.

\begin{definition}[Ultra-antisymmetric elements] 
Let \(A\) be a Banach \(\ast\)-algebra. An element \(a \in A\) with \(\|a\|=1\) is called \emph{ultra-antisymmetric relative to \(A\)} if
\[
\|\,b^* a c + c^* a^* b\,\| < 1 \qquad \text{for all } b,c\in A \text{ with }\|b\|,\|c\|\leq 1.
\]
\end{definition}

\noindent
In trivial situations ultra-antisymmetric elements are abundant. For instance, if $A$ is a Banach $\ast$-algebra with zero product ($ab=0$ for all $a,b\in A$), then every element of norm $1$ is ultra-antisymmetric. The interesting case is when $A$ has nontrivial multiplication. In unital Banach $\ast$-algebras the existence of such elements is more restrictive, and genuine examples require some care. Below we present an example of an element $a$ such that 
\[
b^*ac + c^*a^*b = 0 \qquad \forall\, b,c\in A,
\]
which in particular makes $a$ ultra-antisymmetric. Despite the name, an ultra-antisymmetric element need not be antisymmetric, in the sense that \(a^*=-a\). 

\begin{example}
Consider $a=\begin{psmallmatrix}0&1\\0&0\end{psmallmatrix}\in \Mat_2(\Z_p)$ and let 
$A\subseteq \Mat_2(\Z_p)$ be the unital $*$-subalgebra generated by $1$ and $a$. 
We endow $A$ with the \emph{trivial involution}, i.e.\ $x^*=x$ for all $x\in A$. 
Note that $a^2=0$, and $A=\{\alpha 1+\beta a:\alpha,\beta\in\Z_p\}$ is commutative.

Let $b=\alpha+\beta a$ and $c=\gamma+\delta a$ with $\|b\|,\|c\|\le1$. 
Using $a^2=0$ and the trivial involution we compute
\[
b^* a c + c^* a^* b = b a c + c a b 
= \alpha\gamma\,a + \gamma\alpha\,a 
= 2\alpha\gamma\,a.
\]
Since $\|a\|=1$, it follows that
\[
\|b^* a c + c^* a^* b\| = |2\alpha\gamma|_p.
\]

If $p$ is odd, then $|2|_p=1$, so for instance with $b=c=1$ we obtain 
$\|b^*ac+c^*a^*b\|=1$, which shows that $a$ is not ultra-antisymmetric. 
On the other hand, if $p=2$, then $|2|_2=2^{-1}<1$, so for all 
$\alpha,\gamma\in\Z_2$ with $|\alpha|_2,|\gamma|_2\le1$ we have 
$|2\alpha\gamma|_2\le|2|_2<1$. Hence in this case
\[
\|b^* a c + c^* a^* b\| < 1\qquad \forall\, b,c\in A,\ \|b\|,\|c\|\le1,
\]
which proves that $a$ is ultra-antisymmetric relative to $A$.
\end{example}

In summary, the simple $2\times2$ example above yields a unital Banach $*$-algebra $A$ containing an ultra-antisymmetric element precisely in the case $p=2$.

Next we provide a similar $4\times4$ example, valid for odd primes $p$ such that $-1$ has a square root in $\Z_p$.

\begin{example}\label{ex:not-quasi-C}
Let $p$ be a prime for which an element $i\in\Z_p$ with $i^2=-1$ exists, and consider
\[
a=\begin{pmatrix}
0  & -1 & 0  & -i\\[4pt]
1  &  0 & -i &  0\\[4pt]
0  &  i &  0 & -1\\[4pt]
i  &  0 &  1 &  0
\end{pmatrix}\in \Mat_4(\Z_p),
\]
endowed with the involution given by matrix transpose. Let $A\subseteq \Mat_4(\Z_p)$ be the unital $^*$-subalgebra generated by $a$ and $1$. Then simple computations show that $a^*=-a$ (i.e. $a$ is anti-symmetric) and $a^2=0$, so that
$$A=\{ \alpha 1 + \beta a : \alpha,\beta\in\Z_p\}.$$
Moreover, $a$ is ultra-antisymmetric relative to $A$. Indeed,
for $\alpha,\beta\in\Z_p$ denote $x=\alpha+\beta a\in A$. The involution on $A$ is given by
\[
x^* = \alpha 1 + \beta a^* = \alpha 1 - \beta a,
\]
because the scalar involution is trivial and $a^*=-a$. Let now $b=\alpha+\beta a$ and $c=\gamma+\delta a$ be arbitrary elements of $A$. Using $a^2=0$ and $a^*=-a$ we compute
\begin{align*}
b^* a c
&= (\alpha - \beta a)\,a\,(\gamma + \delta a)
= \alpha\gamma\,a + \alpha\delta\,a^2 - \beta a a\gamma - \beta a a\delta a
= \alpha\gamma\,a,
\\
c^* a^* b
&= (\gamma - \delta a)(-a)(\alpha + \beta a)
= -\gamma\alpha\,a - \gamma\beta\,a^2 + \delta a a\alpha + \delta a a\beta a
= -\gamma\alpha\,a.
\end{align*}
Summing these two expressions yields
\[
b^* a c + c^* a^* b = (\alpha\gamma-\gamma\alpha)\,a = 0.
\]
Therefore
\[
\|b^* a c + c^* a^* b\| = 0 < 1
\]
for every \(b,c\in A\), proving that \(a\) is ultra-antisymmetric relative to $A$.
\end{example}

\begin{remark}\label{rem:not-quasi-C}
Example~\ref{ex:not-quasi-C} also provides a simple counterexample to the quasi-$C^*$-property.  Indeed, for the algebra $A$ generated by the $4\times 4$-matrix $a$ considered above we have $a^*=-a$ and $a^2=0$, and the calculation in the example shows
\[
b^*ac+c^*a^*b=0\qquad\text{for all }b,c\in A.
\]
If $\phi$ is any quasi-state on $A$ (so $\phi$ is continuous, $\phi(x^*)=\phi(x)$ and $\|\phi\|=1$), then
\[
\sup_{\|b\|,\|c\|\le1}|\phi(b^*ac)| \;=\; 0,
\]
because every $b^*ac$ is a scalar multiple of $a$ and, moreover, $\phi(a)=\phi(a^*)=-\phi(a)$, whence $\phi(a)=0$. On the other hand $\|a\|=1$ (for the sup-norm coming from the matrix embedding), so the equality
\[
\sup_{\|b\|,\|c\|\le1}|\phi(b^*ac)|=\|a\|
\]
cannot hold. Therefore $A$ is not a quasi-$C^*$-algebra.
\end{remark}

Some algebras do have ultra-antisymmetric elements, but every algebra can be emerged in a bigger algebra that does not have them:

\begin{theorem}\label{theo:matrices-antisymmetric}
    Let $A$ be a unital Banach $\ast$-algebra. The matrix algebra $\Mat_2(A)$ has no ultra-antisymmetric elements relative to itself.
\end{theorem}
\begin{proof}
    Let $i, j \in \set{1,2}$ and $e_{ij}$ be the elementary matrix in $\Mat_2(A)$. Suppose that
    \begin{align*}
        \tau = \begin{pmatrix}
            a & b \\
            c & d
        \end{pmatrix}
    \end{align*}
    is ultra-antisymmetric. Then \(
        \|e_{11}\tau + \tau^* e_{11}\| < 1\) 
    and
    \begin{align*}
        e_{11}\tau + \tau^* e_{11} = \begin{pmatrix}
            a + a^* & b \\
            b^* & 0
        \end{pmatrix},
    \end{align*}
    so that $\|b\| < 1$. We also have
    \begin{align*}
        \|e_{12}\tau + \tau^* e_{21}\| < 1
    \end{align*}
    and
    \begin{align*}
        e_{12}\tau + \tau^* e_{21} = \begin{pmatrix}
            b^*+c & d \\
            d^* & 0
        \end{pmatrix},
    \end{align*}
    so that $\|d\| < 1$. Similarly, one sees that \(\|a\|<1\) and \(\|c\|<1\), implying that \(\|\tau\|<1\), which is a contradiction.\qedhere
\end{proof}

\begin{theorem} 
Let \(A\) be a unital Banach \(\ast\)-algebra equipped with the \(p\)-adic norm. If \(A\) contains no ultra-antisymmetric elements, then \(A\) is a quasi-\(\Cest\)-algebra. 
\end{theorem}

\begin{proof}
If \(A = 0\), the result is trivial. Suppose \(A \neq 0\). Since \(A\) has a \(p\)-adically complete norm, there is an isomorphism of Banach \(\zP\)-modules
\[
A \cong c_0(X,\zP).
\]

Pick \(a \in A\) with \(a \neq 0\), and write \(u = \lambda a\) for some \(\lambda \in \zP\) such that \(\|u\| = 1\). As \(A\) contains no ultra-antisymmetric elements, there exist \(b, c \in A\) with \(\|b\|, \|c\| \leq 1\) satisfying
\[
\|b^* u c + c^* u^* b\| = 1.
\]

Hence, there exists \(x \in X\) such that
\[
|(b^* u c)(x) + (c^* u^* b)(x)| = 1.
\]

Define
\[
\phi \colon A \to \zP, \quad f \mapsto f(x) + (f^*)(x).
\]
Thus,
\(|\phi(b^* u c)| = 1\), so that \(|\phi(b^* a c)| = \|a\|\), as desired.
\end{proof}

Combining the previous two theorems, we immediately get:

\begin{corollary}
    For every $p$-adically complete unital Banach $*$-algebra $A$, the matrix algebra $\Mat_2(A)$ is a quasi-C*-algebra. In particular, any unital \(p\)-adically complete, involutive Banach \(\zP\)-algebra is Morita equivalent to a quasi-\(C^*\)-algebra. 
\end{corollary}

The following related result shows that a big class of $p$-adic operator algebras does not contain ultra-antisymmetric elements, and hence are quasi-C*-algebras.

\begin{proposition}\label{prop:no-ultra-antisymm}
Let $\HH=\qP(X)$ and suppose $A\subseteq \BUN(\HH)$ is a unital $^*$-subalgebra that contains all finite-rank operators (equivalently, it contains the matrix units $E_{r,s}$ for $r,s\in X$). 
Then $A$ contains no ultra-antisymmetric element: there is no $a\in A$ with $\|a\|=1$ such that
\[
\|b^* a c + c^* a^* b\| < 1\qquad\text{for all }b,c\in A,\ \|b\|,\|c\|\le1.
\]
\end{proposition}

\begin{proof}
Let $a\in A$ with $\|a\|=1$. Write the matrix entries of $a$ in the canonical basis $(\delta_x)_{x\in X}$ as $a_{r,s}=\braket{\delta_r}{a(\delta_s)}$. Since $\|a\|=1$ there exist indices $u,v\in X$ with $|a_{u,v}|_p=1$.

If $X$ is a singleton the statement is trivial, so we may assume that $X$ has at least two elements. Choose distinct indices $i,j\in X$. Put
\[
b=E_{u,i},\qquad c=E_{v,j},
\]
the corresponding rank-one operators (they belong to $A$ by hypothesis and have norm~$\le1$). A direct matrix calculation shows that the $(i,j)$-entry of
\[
T:=b^* a c + c^* a^* b
\]
equals \(a_{u,v}\). Indeed, using $E_{r,s}^*=E_{s,r}$ one checks
\[
(b^* a c)_{i,j}=a_{u,v},\qquad (c^* a^* b)_{i,j}=0
\]
because of the choice \(i\neq j\). Hence
\[
\|T\| \ge |T_{i,j}|_p = |a_{u,v}|_p = 1,
\]
contradicting the ultra-antisymmetry requirement \(\|T\|<1\). Therefore no such \(a\) exists.
\end{proof}

To complete the picture and show that essentially every Banach $*$-algebra can be represented on a quasi-Hilbert space, we need the following technical result.

\begin{proposition}\label{prop:subalgebr} Let $B$ be a Banach $*$-algebra over $\zP$ such that its norm only assumes values in $|\zP|$. Then, there exists a \(p\)-adically complete Banach $\ast$-algebra $A$ and an isometric $*$-homomorphism from $B$ to $A$. Moreover, if $B$ has finite rank, then $A$ can also be choosen to have finite rank. 
\end{proposition}
\begin{proof}
Consider the $\ast$-algebra
\begin{align*}
    C = B[1/p] = \set{p^{k}b \mid k \in \zZ, b\in B} 
\end{align*}
Define the norm
\begin{align*}
    \|p^kb\| = p^{-k} \|b\|.
\end{align*}
This is just the $\qP$-Banach space $B \otimes \qP$. Then the algebra
\begin{align*}
    A =\{x \in C \mid \|x\| \leq 1\}
\end{align*} is a closed $\zP$-subalgebra of $C$, hence is Banach. If $x\in A$ has $\|x\| < 1$, then $x$ is divisible by $p$ in $A$, therefore $A$ is $p$-adically complete. The algebra $B$ can be naturally viewed as a subalgebra of $A$. It is clear that $A$ has finite rank if $B$ does.
\end{proof}

\begin{theorem}[GNS]\label{thm:GNS}
    Every Banach $\ast$-algebra $A$ with $\|A\| \subseteq \vert \zP \vert $ can be isometrically represented inside $\BB(\HH)$ for some quasi-Hilbert space $\HH$.
\end{theorem}
\begin{proof}
    Let $A$ be a Banach $\ast$-algebra. We may assume without loss of generality that \(A\) is unital. If $A$ is not $p$-adically complete, we may use \ref{prop:subalgebr} to embed it inside a $p$-adically complete Banach $\ast$-algebra. Consequently, we may also assume without loss of generality that $A$ is $p$-adically complete. If A has ultra-antisymmetric elements, we may embed it inside \(\Mat_2(A)\) by Theorem \ref{theo:matrices-antisymmetric} to get a $\ast$-algebra containing an isometric copy of $A$ without ultra-antisymmetric elements. As $\Mat_2(A)$ is a quasi-$\Cest$-algebra and thus admits an isometric $\ast$-representation, the result follows.
\end{proof}

\subsection{Finite-rank quasi-Hilbert spaces}

The goal of this section is to prove that essentially any finite rank Banach \(*\)-algebra is a \(p\)-adic operator algebra. To set this up, we first have the following preliminary lemmas:

\begin{lemma}\label{lem:finite-rank}
    Every finite rank (as a $\zP$-module) and $p$-adically complete Banach $\ast$-algebra $A$ with $\|A\| \subseteq |\zP|$ can be isometrically represented on a finite-rank quasi-Hilbert space.
\end{lemma}
\begin{proof}
    If $A$ is finite rank and $p$-adically complete, then so its unitization, so we may suppose $A$ is unital. If $A$ is finite rank and $p$-adically complete, then $\Mat_2(A)$ also is, so we may suppose $A$ contains no ultra-antisymmetric elements. By Theorem \ref{thm:everyone-is-c0}, the underlying \(\zP\)-module of \(A\) is isometrically isomorphic to \(c_0(X,\zP)\). Furthermore, since \(A\) is finite rank, we may take \(X\) to be finite. Setting $X = \set{x_1, \dots, x_n}$, consider the quasi-states
    \begin{align*}
        \phi_k \colon A = c_0(X,\zP) &\to \zP,\\
        f &\mapsto f(x_k) + (f^{\ast})(x_k) ,
    \end{align*}
    for $k \in \set{1, \dots, n}$. Pick $a \in A$ with $a \neq 0$. Write $u = \lambda a$ with $\lambda \in \zP$ and $\|u\| = 1$. As $A$ does not have any ultra-antisymmetric elements, there are $b, c \in A$ with $\|b\|, \|c\| \leq 1$ and
    \begin{align*}
        \|b^*uc + c^*u^*b\| = 1.
    \end{align*}
    Hence there is $x_k \in X$ with
    \begin{align*}
        |(b^*uc)(x_k) + (c^*u^*b)(x_k)| &= 1 \\
        |\phi_k(b^*uc)| &= 1 \\
        |\phi_k(b^*ac)| &= \|a\|.
    \end{align*}
    Let $\HH_k$ be the quasi-Hilbert space associated with $\phi_k$, and let
    \begin{align*}
        \HH = \bigoplus_{1 \leq k \leq n} \HH_k.
    \end{align*}
    Notice that $\HH$ is finite-rank. Just as in the GNS construction, consider
    \begin{align*}
        \pi \colon A &\to \BB(\HH), \\
        a &\mapsto (\pi_{\phi}(a))_{\phi}.
    \end{align*}
    It remains to show that it is isometric. Fix $a \in A$; we already showed that for some $k \in \set{1, \dots, n}$ we have
    \begin{align*}
        \|a\| &= \sup_{\|b\|, \|c\| \leq 1} |\phi_k(b^{\ast}ac)| \\
        &= \sup_{\|b\|, \|c\| \leq 1} |\langle b + I_{\phi_k}, ac + I_{\phi_k} \rangle_{\phi_k}| \\
        &= \sup_{\|b\|, \|c\| \leq 1}|\langle b + I_{\phi_k}, \pi_{\phi_k}(a)(c) \rangle_{\phi_k}| \\
        &= \sup_{ \|c\| \leq 1} \|\pi_{\phi_k}(a)(c)\| \\
        &= \|\pi_{\phi_k}(a)\| \\
        &= \|\pi(a)\|
    \end{align*}
    whence we conclude the claim.
\end{proof}

\begin{lemma}\label{theo:finite-rank-q-Hilb}
    Every nonzero finite rank quasi-Hilbert space $\HH$ admits a linearly independent set of generators $\delta_1, \dots, \delta_n$ with $\langle \delta_i, \delta_j \rangle = 0$ if $i \neq j$ and $\langle \delta_i, \delta_i \rangle = a_i \in \zP$ with $|a_i| = 1$. 
\end{lemma}
\begin{proof}
    Without loss of generality write $\HH = c_0(\set{1, \dots, n}, \zP) = \zP^n$. The pairing on $\HH$ extends naturally to a symmetric nondegenerate bilinear form $b$ on the rationalisation $\HH \hat\otimes \qP$. Since any nondegenerate symmetric bilinear form on a finite-dimensional vector space over a field with characteristic different than $2$ is orthogonalisable \cite{Lang:Algebra}*{XV.Theorem~3.1}, normalizing the basis on $\HH \otimes \qP$ we get a linearly independent set of generators $G = \set{\delta_1, \dots, \delta_n}$ with 
    \begin{align*}
        \langle \delta_i, \delta_j\rangle &= 0 \text{, for }i \neq j, \\
        \langle \delta_i, \delta_i\rangle &= a_i \neq 0.
    \end{align*}
    To show that $|a_i| = 1$, we notice that
    \begin{align*}
        1 = \|\delta_i\| = \sup_{\|f\| \leq 1 } \langle \delta_i, f\rangle = \sup_{\|f\| \leq 1 } \langle \delta_i, \sum_{1 \leq j\leq n} f(j) \delta_j\rangle = \sup_{\|f\| \leq 1 } f(i)a_i = a_i.
    \end{align*}
\vskip -1,5pc
\end{proof}

We will refer to a set of generators as in Theorem \ref{theo:finite-rank-q-Hilb} as an \emph{orthogonal basis}.

\begin{remark}\label{rmk:finite-dim}
    With this result, we see that for a finite-rank quasi-Hilbert space $\HH$ its algebra of bounded operators is isomorphic to $\BB(\HH) = \Mat_n(\zP)$ with the following adjoint operation
    \begin{align*}
        (b_{ij})^{\ast} = (b_{ji}a_ja_i^{-1})
    \end{align*}
    where $a_i$ are the elements in the diagonal of the diagonalization of the pairing associated to $\HH$.
\end{remark}

\begin{lemma}\label{lem:quadratic-extension}
    Suppose that $p \neq 2$. Let $u \in \zP$ be a $p$-adic integer such that the equation $x^2 = u$ does not have a solution on $\zP$. Then, the $\zP$-algebra $\zP[\sqrt{u}]$ with trivial involution and norm given by $\|a + b\sqrt{u}\| = \max\set{|a|_p, |b|_p}$ is a $p$-adic operator algebra.
\end{lemma}
\begin{proof}
    It is a well-known fact from the theory of $p$-adic numbers that every $p$-adic integer with $p$-adic norm $1$ can be written as a sum of $2$ squares. Choose $x, y \in \zP$ such that $u = x^2 + y^2$
    
    Consider the representation
    \begin{align*}
        \phi \colon \zP[\sqrt{u}] &\to \Mat_2(\zP) \\
        a + b\sqrt{u} &\mapsto a \begin{pmatrix}
            1 & 0 \\
            0 & 1
        \end{pmatrix} + b\begin{pmatrix}
            x & y \\
            y & -x
        \end{pmatrix}
    \end{align*}
    It is easy to check that this is a isometric representation.
\end{proof}

\begin{lemma}\label{lem:weird-matrix}
    Let $p \neq 2$ be a prime number. Let $u \in \zP^{\times}$ a $p$-adic integer that is not a square. Consider the Banach $\zP$-algebra $A = \Mat_2(\zP)$ with involution given by
    \begin{align*}
        \begin{pmatrix}
        a & b\\
        c & d
    \end{pmatrix}^{\ast} = \begin{pmatrix}
        a & cu^{-1} \\
        bu & d
    \end{pmatrix}.
    \end{align*}
    The algebra $A$ is a $p$-adic operator algebra.
\end{lemma}
\begin{proof}
    By Lemma \ref{lem:quadratic-extension} we know that $\zP[\sqrt{u}]$ with trivial involution is a $p$-adic operator algebra. Consider the tensor product
    \begin{align*}
      \Mat_2(\zP[\sqrt{u}]) = \Mat_2(\zP) \otimes \zP[\sqrt{u}],
    \end{align*}
    with $\Mat_2(\zP)$ carrying the usual $p$-adic operator algebra structure. We know that $\Mat_2(\zP[\sqrt{u}])$ is a $p$-adic operator algebra. To represent $A$ on $\Mat_2(\zP[\sqrt{u}])$, consider
    \begin{align*}
        \phi \colon A &\to \Mat_2(\zP[\sqrt{u}]) \\
        \begin{pmatrix}
            a & b \\
            c & d
        \end{pmatrix} &\mapsto \begin{pmatrix}
            a & b\sqrt{u} \\
            c\sqrt{u^{-1}} & d
        \end{pmatrix}.
    \end{align*}
    One quickly checks that $\phi$ is a $\ast$-homomorphism and it is easy to check that it is isometric as every algebra involved is $p$-adically complete and a matrix on the domain is divisible by $p$ if, and only if is divisible by $p$ on the right.
\end{proof}

\begin{lemma}\label{lem:weirder-matrix}
    Let $p \neq 2$ be a prime number. Let $u \in \zP^{\times}$ a $p$-adic integer that is not a square. Let $\HH$ be a quasi-Hilbert space generated by $w_1, w_2, \dots, w_{2n}$ with 
    \begin{enumerate}
        \item $\langle w_k, w_l \rangle = 0$ for $k \neq l$.
        \item $\langle w_k, w_k \rangle = 1$ for $1 \leq k \leq n$ and $\langle w_l, w_l \rangle = u$ for $n < l \leq 2n$.
    \end{enumerate}
    Then the Banach $\ast$-algebra $\BB(\HH)$ is a $p$-adic operator algebra.
\end{lemma}
\begin{proof}
    Following \ref{rmk:finite-dim}, it follows that $\BB(\HH)$ is isomorphic to $\Mat_{2n}(\zP)$ (as a Banach algebra) with involution given by
    \begin{align}\label{eq:involution-weird}
        \begin{pmatrix}
            A_{00} & A_{01} \\
            A_{10} & A_{11}
        \end{pmatrix}^{\ast} = \begin{pmatrix}
            A_{00}^t & u^{-1}A_{10}^t \\
            uA_{01}^t & A^t_{11}
        \end{pmatrix}
    \end{align}
    dividing in blocks with size $n \times n$. 

    Let $\Mat_2(\zP)^w$ be the $p$-adic operator algebra of Lemma \ref{lem:weird-matrix}. Consider the map
    \begin{align*}
        \phi \colon \BB(\HH) &\to \Mat_2(\zP)^w \otimes \Mat_n(\zP), \\
        \begin{pmatrix}
            A & B \\
            C & D
        \end{pmatrix} &\mapsto e_{11} \otimes A + e_{12} \otimes B + e_{21} \otimes C + e_{22} \otimes D
    \end{align*}
    The map $\phi$ is a homomorphism of $\zP$-algebras because of classical matrix theory, it preserves involution thanks to \eqref{eq:involution-weird}, it is isometric because every algebra involved is $p$-adically complete and a matrix is divisible by $p$ on the domain if, and only if is divisible by $p$ on the codomain. It is surjective because both algebras have rank $(2n)^2$ as $\zP$-modules and $\phi$ is surjective.
\end{proof}

\begin{theorem}
    Let $p \neq 2$ be a prime number. Let $\HH$ be a finite-rank quasi-Hilbert space, then $\BB(\HH)$ is a $p$-adic operator algebra.
\end{theorem}
\begin{proof}
    Let $w_1, w_2, \dots, w_n$ be an orthogonal basis for $\HH$.  Consider the multiplicative group
    \begin{align*}
        (\zP^{\times})^2 = \set{a^2 \in \zP^{\times} \mid a \in \zP^{\times}}.
    \end{align*}
    Then the quotient group \(\zP^{\times}/ (\zP^{\times})^2\)
    has order $2$ for $p \neq 2$. Let $u \in \zP^{\times}$ be such that it does not have a square root. Then 
    \begin{align*}
        \zP^{\times}/ (\zP^{\times})^2 = \set{\overline{1}, \overline{u}}
    \end{align*}
    Thus, every $p$-adic integer with absolute value equal to $1$ is a square or a square times $u$. For each $k \in \set{1, \dots, n}$, let $d_k \in \zP^{\times}$ such that $d_k^2 = \langle w_k, w_k\rangle$ if $\langle w_k, w_k\rangle$ is a square and $d_k^2 = u\langle w_k, w_k\rangle$ if $\langle w_k, w_k\rangle$ is not a square. Consider the new orthogonal basis defined by the elements $v_k := w_k/d_k$. This orthogonal basis satisfies that $\langle v_k , v_k\rangle$ is either $1$ or $u$. By reordering the basis we can get a basis $v_1, \dots, v_n$ such that 
    \begin{enumerate}
        \item $\langle v_k, v_k \rangle = 1$ for $k \leq m$ and $\langle v_l, v_l \rangle = u$ for $m < l \leq n$.
        \item $\langle v_k, v_l \rangle = 0$ for $k \neq l$.
    \end{enumerate}
    Following Remark \ref{rmk:finite-dim}, the algebra $\BB(\HH)$ carries the following involution
    \begin{align*}
        (T_{kl})^{\ast} = (T_{kl}^{\ast})
    \end{align*}
    with
    \begin{align*}
        T_{kl}^{\ast} = \begin{cases}
            T_{lk} \text{, if }k, l\leq m \text{ or }k, l>m , \\
            uT_{lk} \text{, if } k \leq m <l, \\
            u^{-1}T_{lk} \text{, if } l \leq m < k
        \end{cases}
    \end{align*}
    Consider the Banach $\ast$-algebra $\BB(\HH) \otimes \zP[\sqrt{u}]\cong \Mat_n(\Zp[\sqrt{u}])$, where $\zP[\sqrt{u}]$ carries the trivial involution. For each $r \in \set{1, \dots, n}$ consider the quasi-state
    \begin{align*}
        \tau_r \colon \BB(\HH) \otimes \zP[\sqrt{u}] &\to \zP ,\\
        (a_{kl} + b_{kl}\sqrt{u}) &\mapsto a_{rr}.
    \end{align*}
    The quasi-state $\tau_r$ induces a bilinear map
    \begin{align*}
        \langle (a_{kl} + b_{kl}\sqrt{u}), (c_{kl} + d_{kl}\sqrt{u})\rangle_{\tau_r} &= \tau_r((a_{kl} + b_{kl}\sqrt{u})^*(c_{kl} + d_{kl}\sqrt{u})) 
    \end{align*}
    The quasi-state $\tau_r$ yields a quasi-Hilbert space $\HH_{\tau_r}$, which can be canonically identified with space $C_r$ of $r$-column on $\BB(\HH) \otimes \zP[\sqrt{u}] \cong \Mat_n(\Zp[\sqrt{u}])$. Let us denote by  $e_1,\dots, e_n$ canonical column elements of $C_r$, with $e_i$ having entry $1$ at $i$ and zero otherwise. 
    
    \und{For $r > m$:} Consider the basis $e_1$, $e_2$,..., $e_n$, $\sqrt{u}e_1$, $\sqrt{u}e_2$,..., $\sqrt{u}e_m$, $\sqrt{u^{-1}}e_{m+1}$,..., $\sqrt{u^{-1}}e_n$  of $C_r$. This is orthogonal basis for $C_r$ and for one half of those vectors we have $\langle v, v \rangle = 1$ and the other half $\langle v , v \rangle = u^{-1}$. Therefore $\BB(\HH_{\tau_r}) \cong \BB(C_r)$ is a $p$-adic operator algebra by Lemma \ref{lem:weirder-matrix}.

    \und{For $r \leq m$:} Consider the basis $e_1$, $e_2$,..., $e_m$, $ue_{m+1}$,..., $ue_{n}$, $\sqrt{u}e_1$, $\sqrt{u}e_2$,..., $\sqrt{u}e_m$, $\sqrt{u^{-1}}e_{m+1}$,..., $\sqrt{u^{-1}}e_n$  of $C_r$. This is orthogonal basis for $C_r$ and for one half of those vectors we have $\langle v, v \rangle = 1$ and the other half $\langle v , v \rangle = u$. Therefore $\BB(\HH_{\tau_r}) \cong \BB(C_r)$ is a $p$-adic operator algebra again by Lemma \ref{lem:weirder-matrix}.
    
    By the GNS construction, we have a representation
    \begin{align*}
        \pi \colon \BB(\HH) \otimes \zP[\sqrt{u}] &\to \prod_{1 \leq r \leq n} \BB(\HH_{\tau_r}) \\
        a &\mapsto (\pi_{\tau_r}(a))_r,
    \end{align*}
    which is easily seen to be an isometric inclusion. Indeed, let $E_{kl}$ be the elementary matrix with $1$ at the entry $(k, l)$ and $0$ otherwise. Then we have \(\pi(a) = 0\) if and only if \(\pi_{\tau_r}(a) = 0\) for all \(r\), so that \(a\cdot E_{lr} = 0\) for all \(r\) and \(l\) implying that \(a = 0\).

    Therefore $\BB(\HH) \otimes \zP[\sqrt{u}]$ is a $p$-adic operator algebra and as $\BB(\HH)$ is a closed $\ast$-subalgebra it is also a $p$-adic operator algebra, as required.
\end{proof}

\begin{corollary}\label{cor:finite-rank-Banach-alg}
For $p\not=2$, every finite-rank Banach $\ast$-algebra is a $p$-adic operator algebra.
\end{corollary}
\begin{proof}
Let $A$ be a finite-rank Banach $*$-algebra. By Proposition~\ref{prop:subalgebr}, we may assume that $A$ is $p$-adically complete. Then by Lemma~\ref{lem:finite-rank} $A$ can be isometrically represented on $\BB(\HH)$, for some finite-rank quasi-Hilbert space $\HH$. The result now follows from the previous theorem. 
\end{proof}

We believe that the above result is also true for $p=2$, but the proof would require a different idea.

\subsection{Residually finite-rank Banach $*$-algebras}

As a consequence of our analysis of finite-rank quasi-Hilbert space in the previous section, we may now show that several new families of Banach $\ast$-algebras over $\zP$ are $p$-adic operator algebras.

\begin{definition}[RFR $p$-adic $\ast$-algebras]
    Let $A$ be a Banach $\ast$-algebra over $\zP$. Say that $A$ is \textit{residually finite-rank} if it can be isometrically represented in a product
    \begin{align*}
        \phi \colon A \to \prod_{i \in I} \BB(\HH_i),
    \end{align*}
    where $\set{\HH_i}_{i \in I}$ is a family of finite-dimensional quasi-Hilbert spaces.
\end{definition}

Since the category of $p$-adic operator algebras is closed under direct products (see \cite{BGM:padicI}), the following result is an immediate consequence of Corollary~\ref{cor:finite-rank-Banach-alg}.

\begin{corollary}\label{cor:RFR-p-adic-op}
    Let $p \neq 2$ be a prime number. Every residually finite-rank Banach $\ast$-algebra over $\zP$ is a $p$-adic operator algebra.
\end{corollary}

\begin{example}[Permutation Tate Algebra]
    Let us apply the $p$-adic GNS theorem and its corollaries. Let $p \neq 2$ be a odd prime number. Consider the Tate algebra $\zP\langle X, Y \rangle$ with involution induced by the formula $X^* = Y$. Consider the family of quasi-states
    \begin{align*}
        \tau_n \colon \zP\langle X, Y \rangle &\to \zP, \\
        \sum_{(a, b) \in \nN^2} r_{a, b} X^aY^b&\to r_{n, n}
    \end{align*}
    for each $n \in \nN$. The quasi-state $\tau_n$ induces the symmetric bilinear form
    \begin{align*}
        \left\langle  \sum_{(a, b) \in \nN^2} r_{a, b}X^aY^b,  \sum_{(c, d) \in \nN^2} s_{c, d} X^cY^d\right\rangle_n = \sum_{a + c = b + d = n} r_{a, b}s_{c, d}.
    \end{align*}
    Each of these pairings are degenerate. One can calculate that
    \begin{align*}
        I_n &:= \set{f \in \zP\langle X, Y \rangle \mid \langle f , g \rangle_n = 0 \ \forall g \in \zP\langle X, Y \rangle} \\
        &= (X^{n+1}, Y^{n+1}) .
    \end{align*}
    Therefore, the associated quasi-Hilbert space is
    \begin{align*}
        \HH_n := \frac{\zP\langle X, Y \rangle}{(X^{n+1}, Y^{n+1})},
    \end{align*}
    which has finite-rank. For each $n \in \nN$ we have a representation
    \begin{align*}
        \pi_n \colon \zP\langle X, Y \rangle &\to \BB(\HH_n), \\
        \pi_n(f)(g + I_n) &= fg + I_n.
    \end{align*}
    These representations satisfy that given $f = \sum_{(a, b) \in \nN^2} r_{a, b}X^aY^b$, there exists $n \in \nN$ such that $\|f\| = \|\pi_n(f)\|$. To see this, suppose that $\|f\| = |r_{a_0, b_0}|_p$ and let $n = \max\set{a_0, b_0}$. We have
    \begin{align*}
        \|\pi_n(f)(1 + I_n)\| = \|f + I_n\| = \max_{a, b \leq n} |r_{a, b}|_p = |r_{a_0, b_0}|_p = \|f\|.
    \end{align*}
    Hence $\|\pi_n(f)\| \leq \|f\|$ and $\pi_n$ is contractive by construction. Thus the direct sum of the representations $\pi_n$ is isometric. Hence $\zP\langle X, Y \rangle$ is RFR and thus a $p$-adic operator algebra.
\end{example}

Finally, we show that the unit ball of a reduced affinoid \(\qP\)-algebra is a \(p\)-adic operator algebra. Here by an affinoid \(\qP\)-algebra, we mean the quotient \(A = \qP\langle x_1,\dotsc,x_n\rangle / I\) of the Tate algebra of convergent power series on a closed unit polydisc. We view this as a Banach algebra with respect to the norm \[\|f\| = \max_{\Mm \subseteq A} \|f + \Mm\|.\] 

\begin{theorem}
    Suppose $p \neq 2$. Let $A$ be a reduced affinoid $\qP$-algebra, and \(B\) its unit ball which we equip with the trivial involution. Then $B$ is a $p$-adic operator algebra.
\end{theorem}
\begin{proof}
    We have an isometry
    \begin{align*}
        \phi \colon A &\to \prod_{\Mm \subset A} \frac{A}{\Mm}, \\
        f &\mapsto (f + \Mm)_{\Mm}.
    \end{align*}
    By the Hilbert's Nullstellensatz, the algebra $A/\Mm$ has finite dimension over $\qP$. Now consider the Banach $\zP$-algebra
    \begin{align*}
        B_\Mm = \set{f + \Mm \in A/\Mm \mid \|f + \Mm\| \leq 1}
    \end{align*}
    with trivial involution; this algebra has finite rank over $\zP$ and is therefore is a $p$-adic operator algebra by Corollary \ref{cor:RFR-p-adic-op}. Restricting the map $\phi$ we get an isometry
    \begin{align*}
        \phi|_{B} \colon B \to \prod_{\Mm \subset A} B_\Mm,
    \end{align*}
    showing that $B$ is a $p$-adic operator algebra.
\end{proof}

\section{Homotopy analytic \(K\)-theory for commutative \(p\)-adic operator algebras}\label{sec:kh}

In this section, we compute the \(K\)-theory of continuous functions \(C(X,\zP)\) on a compact Hausdorff space with values in \(\zP\). To this end, we recall the following general approach due to Dustin Clausen. Denote by \(\mathsf{CHaus}\) denote the full subcategory of topological spaces of compact Hausdorff spaces. We say that a functor \(F \colon \mathsf{CHaus}^{\mathrm{op}} \to \mathsf{Sp}\) valued in the \(\infty\)-category of spectra satisfies 

\begin{itemize}
    \item \textit{closed descent}: if for any pair \(K\), \(L \subseteq X\) of closed subsets of a compact Hausdorff space \(X\), we have  a pullback square \[
    \begin{tikzcd}
        F(K \cup L) \arrow{r}{} \arrow{d}{} & F(K) \arrow{d}{} \\
        F(L) \arrow{r}{} & F(K \cap L)
    \end{tikzcd}
    \] in spectra. 
    \item \textit{profinite descent}: if for any cofiltered limit \(X = \lim X_i\), we get a filtered colimit \(F(X) \simeq colim F(X_i)\) of spectra. 
\end{itemize}

Let \(\mathsf{Fun}^{cpd}(\mathsf{CHAUS}^\mathrm{op},\mathsf{Sp})\) denote the category of functors satisfying closed and profinite descent.

\begin{theorem}[Clausen]\label{thm:Clausen}
    Let \(F \colon \mathsf{CHAUS}^\mathrm{op} \to \mathsf{Sp}\) satisfying closed descent and profinite descent. Then \(F\) is determined by its value on the point. More precisely, we have an equivalence of categories \[\mathsf{Fun}^{cpd}(\mathsf{CHAUS}^\mathrm{op},\mathsf{Sp}) \to \mathsf{Sp}, \quad F \mapsto F(*),\] with inverse given by \[E 
    \mapsto (X \mapsto \Gamma(X, \underline{E})),\] with \(\underline{E}\) the constant sheaf at 
    \(E\).
\end{theorem}

Now consider the analytic \(K\)-theory spectrum \(KH^\an \colon \mathsf{Alg}(\mathsf{CBorn}_{\Zp}^{\mathrm{tf}}) \to \mathsf{Sp}\) defined in \cite{mukherjee2025nonarchimedea}. When restricted to \(p\)-adically complete Banach \(\zP\)-algebras, this yields a weak equivalence \[KH^\an(A) \simeq KH(A/p)\] of spectra with Weibel's homotopy algebraic \(K\)-theory of the reduction mod \(p\) (see \cite{mukherjee2025nonarchimedea}*{Theorem 6.8}).

\begin{proposition}\label{prop:K-theory}
    The functor \(KH^\an(C(-,\zP)) \colon \mathsf{CHAUS}^\mathrm{op} \to \mathsf{Sp}\) satisfies the requirements of Theorem \ref{thm:Clausen}. Consequently, we have \[KH^\an(C(X,\zP)) \simeq \Gamma(X, K(\fP)).\]
\end{proposition}

\begin{proof}
For a pair of closed subsets \(K\), \(L\) in \(X\), we have the following pullback diagram 
\[
\begin{tikzcd}
    C(K \cup L, \zP) \arrow{r}{} \arrow{d}{} & C(K, \zP) \arrow{d}{} \\
    C(L,\zP) \arrow{r}{} & C(K \cap L, \zP).
\end{tikzcd}
\] in the category of Banach \(\zP\)-algebras. For a compact Hausdorff space, denote by \(KH^\an(X) = KH^\an(C(X,\zP))\), and \(KH(X) = KH(C(X,\fP))\); we are required to show that \[KH^\an(K \cup L) \simeq KH^\an(K) \times_{KH^\an(K \cap L)} KH^\an(L).\] By \cite{mukherjee2025nonarchimedea}*{Theorem 6.8}, the canonical map \[KH^\an(K \cup L) \to KH^\an(K) \times_{KH^\an(K \cap L)} KH^\an(L)\] is equivalent to \[KH(K \cup L) \to KH(K) \times_{KH(K \cap L)} KH(L),\] which is an equivalence by \cite{weibel1989homotopy}*{Corollary 2.2}. 

For profinite descent, we observe that for a cofiltered limit \(X = \varprojlim_{i \in I} X_i\), we get a filtered colimit \(C(X,\zP) = \varinjlim_{i \in I} C(X_i, \zP)\), and since analytic \(K\)-theory commutes with filtered colimits (for the same reason that homotopy algebraic \(K\)-theory does), we get a filtered colimit \[KH^\an(C(X,\zP)) \simeq \varinjlim_i KH^\an(C(X_i, \zP))\] of spectra. This checks the conditions of Theorem \ref{thm:Clausen}. 
\end{proof}

\section{Non-homotopy invariant \(K\)-theory for \(p\)-adic operator algebras}\label{sec:nonhomotopy}

In this section, we consider invariants of \(p\)-adic operator algebras that are non-homotopy invariant. Recall that in the complex \(C^*\)-algebraic setting, by \cite{higson1988algebraic}, algebraic and topological \(K\)-theory coincide for compact stable \(C^*\)-algebras. In particular, algebraic \(K\)-theory is homotopy invariant on stable \(C^*\)-algebras. The analogous situation substituting topological \(K\)-theory for the homotopy analytic \(K\)-theory constructed in \cite{mukherjee2025nonarchimedea} is as far away from algebraic \(K\)-theory as possible. Indeed, even in the case of the trivial \(p\)-adic operator algebra \(\zP\), we have \[KH^{\mathrm{an}}(\zP) \simeq KH(\fP)\simeq K(\fP),\] which is not \(K(\zP)\). An intermediate invariant (or better put, a refinement) is to consider the inverse limit \[K^{\mathrm{cont}}(\hat{A}) := \varprojlim K(A/p^n)\] of the algebraic \(K\)-theory spectra of the reductions mod \(p^n\) for each n; this is called \emph{continuous \(K\)-theory}. In the context of formal schemes, recent work due to Efimov interprets the continuous \(K\)-theory spectrum of a \(p\)-adically complete commutative Noetherian ring \(A\) as the \(K\)-theory spectrum of a certain \emph{dualisable category} of nuclear \(A\)-modules. Using very similar arguments, the description of continuous \(K\)-theory as the \(K\)-theory of nuclear \(A\)-modules carries over to noncommutative \(\zP\)-algebras, thereby providing an invariant of \(p\)-adic operator algebras that is computed in terms of the algebraic \(K\)-theory spectra of the reductions mod \(p^n\). 

To describe the relevant spectra, some background is in order. Recall that invariants such as algebraic and homotopy algebraic \(K\)-theory are not just invariants of rings but rather invariants of small stable \(\infty\)-categories. Denote the \(\infty\)-category of small, idempotent complete, stable \(\infty\)-categories by \(\mathbf{Cat}_{\infty}^{\mathrm{perf}}\). In this generality, the invariants just mentioned are examples of \emph{localising invariants}, that is, functors \[E \colon \mathbf{Cat}_{\infty}^{\mathrm{perf}} \to \mathbf{D}\] into stable \(\infty\)-categories that preserve Verdier sequences. By \cite{efimov2024k}, any such invariant extends uniquely to a localising invariant of \emph{dualisable} categories 

\[
\begin{tikzcd}
    \mathbf{Cat}_{\infty}^{\mathrm{perf}} \arrow{r}{E} \arrow[swap]{d}{\mathsf{Ind}(-)} & \mathbf{D} \\
    \mathbf{Cat}_{\infty}^{\mathrm{dual}} \arrow[swap]{ur}{E^{\mathrm{cont}}}&
\end{tikzcd}
\] where \(\mathbf{Cat}_{\infty}^{\mathrm{dual}}\) denotes the \(\infty\)-category of dualisable categories. Applying this to our context, let \(A\) be a torsionfree \(\zP\)-algebra. Then the derived \(\infty\)-categories of (left) \(A/p^n\)-modules \[\cdots \to \mathbf{D}(A/p^{n+1}) \to \mathbf{D}(A/p^n) \to \cdots \to \mathbf{D}(A/p)\] is an inverse system of dualisable categories for \(n \geq 1\).

Now consider the category \(\mathbf{D}(\mathsf{Ban}_{\zP}^{\leq1})^\circ\) of derived convex \(\zP\)-modules: its objects are chain complexes that are quasi-isomorphic to chain complexes that are termwise \(p\)-torsionfree and are Banach \(\zP\)-modules with respect to the \(p\)-adic norm (see \cite{kelly2025localising}*{Lemma 5.22}). This is a monoidal subcategory of \(\mathbf{D}(\mathsf{Ban}_{\zP}^{\leq 1})\) with respect to the completed projective tensor product. In particular, the underlying \(\zP\)-module of a \emph{bornological} \(p\)-adic operator algebra - that is, a \(p\)-adic operator algebra with respect to the \(p\)-adic norm - is contained as a full subcategory inside chain complexes in \(\mathbf{D}(\mathsf{Ban}_{\zP}^{\leq1})^\circ\) that are concentrated in degree zero.  

\begin{proposition}\label{prop:inverse-lim}
    Let \(A\) be a bornological \(p\)-adic operator algebra. We have an equivalence 
    \[\mathbf{Mod}_A(\mathbf{D}(\mathsf{Ban}_{\zP}^{\leq1})^\circ) \to \varprojlim \mathbf{D}(A/p^n)\] of \(\infty\)-categories.  
\end{proposition}

\begin{proof}
    By \cite{kelly2025localising}*{Theorem 5.24}, we have \[\mathbf{D}(\mathsf{Ban}_{\zP}^{\leq1})^\circ) \simeq \varprojlim \mathbf{D}(\mathbb{Z}/p^n).\] Since \(A\) is a bornological \(p\)-adic operator algebra, it is an algebra object in \(\mathbf{D}(\mathsf{Ban}_{\zP}^{\leq1})^\circ\). Applying the functor \(- \otimes_{\zP}^L A/p^n\) on the left hand side of the equivalence above yields functors into \(\mathbf{D}(A/p^n)\) for each \(n\), so that we obtain a canonical functor \[\mathbf{D}(\mathsf{Ban}_A^{\leq 1})^\circ \simeq \mathbf{Mod}_A(\mathbf{D}(\mathsf{Ban}_{\zP}^{\leq1})^\circ) \to \varprojlim \mathbf{D}(A/p^n).\] Now use \cite{kelly2025localising}*{Theorem 5.24} for \(A\) in place of \(\zP\), which is in particular a \(p\)-adic Banach ring.
\end{proof}

By \cite{kelly2025localising}*{Corollary 5.23}, the category \(\mathbf{Mod}_A(\mathbf{D}(\mathsf{Ban}_{\zP}^{\leq 1})^\circ)\) is a reflective subcategory of \(\mathbf{D}(\mathsf{Ban}_R^{\leq 1})\), so that it is presentable. It is, however, not yet a dualisable category; this needs to be forced. Given a closed symmetric monoidal stable presentable \(\infty\)-category \(\mathbf{C}\), there is a rigid (in particular dualisable) category \(\mathbf{C}^{\mathrm{rig}}\) and a functor \(\mathbf{C}^{\mathrm{rig}}\) exhibiting the category of rigid categories as a coreflective subcategory of the category \(\mathbf{CAlg}(\mathsf{Pr}_{st}^L)\) of commutative algebra objects in the category of presentable stable \(\infty\)-categories with left adjoint functors as morphisms, equipped with the Lurie tensor product. We do not provide further details of this construction, and refer the reader to \cite{nkp}*{Theorem 4.4.17}. In our case, the rigidification simplifies considerably to yield the following:

\begin{corollary}\label{cor:rigid-nuc}
    Let \(A\) be a bornological \(p\)-adic operator algebra. We have an equivalence \[\mathbf{Mod}_A(\mathbf{D}^(\mathsf{Ban}_{\zP}^{\leq1})^\circ)^{\mathrm{rig}} \to \varprojlim^{\mathrm{dual}} \mathbf{D}(A/p^n),\] where the right hand side denotes the inverse limit taken in the category of dualisable categories. 
\end{corollary}
\begin{proof}
    Using that the categories \(\mathbf{D}(A/p^n)\) are already rigid, applying the rigidification functor to the equivalence in Proposition \ref{prop:inverse-lim}, and using that the rigidification functor is right adjoint to the inclusion of rigid categories in \(\mathbf{CAlg}(\mathsf{Pr}_{st}^L)\), we get an inverse limit of dualisable categories taken in \(\mathbf{Cat}_{\infty}^{\mathrm{dual}}\).  
\end{proof}

Following \cite{nkp}, we denote the dualisable inverse limit \[\widetilde{\mathbf{Nuc}}(\hat{A}) \defeq \varprojlim^{\mathrm{dual}} \mathbf{D}(A/p^n)\] and refer to the resulting category as the category of \emph{nuclear} \(A\)-modules. Note that this definition makes sense for any \(\zP\)-algebra; the conclusion of Corollary \ref{cor:rigid-nuc} is only claimed for bornological \(p\)-adic operator algebras. For what follows, we abbreviate \(\mathsf{Nuc}(A) \defeq \mathbf{Mod}_A(\mathbf{D}^(\mathsf{Ban}_{\zP}^{\leq1})^\circ)^{\mathrm{rig}}\).

For a bornological \(p\)-adic operator algebra \(A\), we now define \[K^{\mathrm{an}}(A) \defeq K^{\mathrm{cont}}(\mathsf{Nuc}(A))\] as in \cites{efimov2025localizing,kelly2025localising}. Note that to extend this definition to all \(p\)-adic operator algebras (and not just the bornological ones), we need to take \(A\)-modules internal to the larger category \(\mathsf{CBorn}_R\) of complete bornological \(\zP\)-modules. We expect that the rigidification of the derived \(\infty\)-category of complete bornological \(A\)-modules is equivalent to the category \(\widetilde{\mathsf{Nuc}}(A)\); this will be the subject of a forthcoming article. 

\begin{theorem}
   Let \(A\) be a bornological \(p\)-adic operator algebra. Then we have an equivalence \[K^{\mathrm{an}}(A) \simeq \varprojlim K(A/p^n)\] of spectra.
\end{theorem}

\begin{proof}
    Under the hypotheses on \(A\), by Corollary \ref{cor:rigid-nuc}, we have an equivalence \(\mathsf{Nuc}(A) \simeq \widetilde{\mathsf{Nuc}}(A)\). We need to check that the inverse system \((\mathbf{D}(A/p^n))_{n \geq 1}\) is a Mittag-Leffler system in the sense of \cite{efimov2025localizing}*{Definition 5.1}, which is essentially \cite{efimov2025localizing}*{Proposition 5.23}; we simply show that a very similar proof works in the noncommutative setting. Indeed, the Mittag-Leffler condition is equivalent to verifying that for each \(k\geq 1\), the inverse system \((A/p^k \otimes_{A_{p^n}}^L A/p^k)_{n>k}\) is essentially constant in the derived \(\infty\)-category of \(A/p^k\)-bimodules. A straightforward \(\mathrm{Tor}\)-computation shows that \[A/p^k \otimes_{A_{p^n}}^L A/p^k \simeq (\cdots \overset{0}\to A/p^k \overset{0}\to A/p^k \cdots \overset{0}\to A/p^k),\] using the \(2\)-periodic resolution \(A/p^n \overset{p^k}\to A/p^n \overset{p^{n-k}}\to A/p^n \overset{p^k}\to A/p^n \to A/p^k \to 0\) of \(A/p^k\) by free (right) \(A/p^n\)-modules. Using the same resolution, the structure maps \(A/p^k \otimes_{A/p^{n+1}}^L A /p^k \to A/p^k \otimes_{A/p^{n}}^L A /p^k \) are given by 
    \[
    \begin{tikzcd}
      \cdots \arrow{r}{0} &  A/p^k \arrow{r}{0} \arrow{d}{0} & A/p^k \arrow{r}{0} \arrow{d}{0} & A/p^k \arrow{r}{0} \arrow{d}{1} & A/p^k \arrow{d}{1}\\
      \cdots \arrow{r}{0} &  A/p^k \arrow{r}{0} \arrow{d}{0} & A/p^k \arrow{r}{0} \arrow{d}{0} & A/p^k \arrow{r}{0} \arrow{d}{1} & A/p^k \arrow{d}{1} \\
      \cdots \arrow{r}{0} &  A/p^k \arrow{r}{0}   & A/p^k \arrow{r}{0} & A/p^k \arrow{r}{0} & A/p^k, 
    \end{tikzcd}
    \] whose homotopy inverse limit is the chain complex \(A/p^k[0] \oplus A/p^k[1]\), which manifestly lies in \(\mathbf{Perf}(A/p^k)\) and is essentially constant as a pro-system. Consequently, we have a Mittag-Leffler system of \(A\)-linear categories. The conclusion now follows from \cite{efimov2025localizing}*{Corollary 6.2}.
\end{proof}

s

\end{document}